\documentclass[12pt,centertags,oneside]{amsart}
\usepackage{amsmath,amstext,amsthm,a4,amssymb,amscd}
\usepackage{mathrsfs,dsfont}
\usepackage{fancyhdr}
\usepackage{epsf}
\usepackage{mathtools}
\usepackage
{hyperref}
\hypersetup{
    colorlinks = true,
    linkcolor={magenta},
    citecolor={magenta},
     }
\usepackage{charter}
\usepackage{typearea}
\usepackage{xcolor}
\usepackage{listings}
\usepackage{enumitem}
\usepackage[utf8]{inputenc}

\usepackage[a4paper,width=16.2cm,top=2.5cm,bottom=2cm,
footskip=1cm]{geometry}

\numberwithin{equation}{section}
\newtheorem{theorem}{Theorem}[section]
\newtheorem{lemma}[theorem]{Lemma}
\newtheorem{proposition}[theorem]{Proposition}

\theoremstyle{definition}
\newtheorem{definition}[theorem]{Definition}
\newtheorem{example}{Example}

\newtheorem{remark}[theorem]{Remark}

\numberwithin{equation}{section}

\newcommand{\field}[1]{\mathbb{#1}}
\newcommand{\Z}{\field{Z}}
\newcommand{\R}{\field{R}}
\newcommand{\C}{\field{C}}
\newcommand{\N}{\field{N}}

\newcommand{\cali}[1]{\mathscr{#1}}
\newcommand{\cC}{\cali{C}}

\DeclareMathOperator{\Ker}{Ker}

\DeclareMathOperator{\Id}{Id}
\DeclareMathOperator{\supp}{supp}

\DeclareMathOperator{\SL}{SL}
\DeclareMathOperator{\Sp}{Sp}
\DeclareMathOperator{\GL}{GL}
\DeclareMathOperator{\SU}{SU}
\DeclareMathOperator{\M}{M}

\newcommand{\spin}{$\text{spin}^c$ }
\newcommand{\norm}[1]{\lVert#1\rVert}

\newcommand{\om}{\omega}

\newcommand{\wi}{\widetilde}

\newcommand{\comment}[1]{}

\allowdisplaybreaks

\begin{document}

\title{Bergman kernels and Poincar\'e series}

\date{\today}

\author{Louis Ioos}
\address{CY Cergy Paris Université, 95300 Cergy-Pontoise,
France}
\email{louis.ioos@cyu.fr}
\thanks{L.\ I.\ was partially 
supported by DIM-R\'egion Ile-de-France, by
the European Research Council Starting grant 757585
and by the ANR-23-CE40-0021-01 JCJC
project QCM}
\author{Wen Lu}
\address{School of Mathematics and Statistics,
\& Hubei Key Laboratory of Engineering Modeling 
\newline
    \mbox{\quad} and Scientific Computing,
 Huazhong University of Science and Technology, 
 Wuhan 430074, \newline
    \mbox{\quad}\,China}
\email{wlu@hust.edu.cn}
\thanks{W.\ L.\ supported by National Natural Science Foundation
of China (Grant Nos. 11401232, 11871233)}
\author{Xiaonan Ma}
\address{Chern Institute of Mathematics and LPMC, Nankai University, Tianjin 30071, P.R. China}
\email{xiaonan.ma@nankai.edu.cn}
\thanks{X.\ M.\ was partially supported by
Nankai Zhide Foundation, ANR-14-CE25-0012-01,  and
funded through the Institutional Strategy of
the University of Cologne within the German Excellence Initiative.}
\author{George Marinescu}
\address{Universit{\"a}t zu K{\"o}ln,  Mathematisches Institut,
    Weyertal 86-90,   50931 K{\"o}ln, Germany
    \newline
    \mbox{\quad}\,Institute of Mathematics `Simion Stoilow',
	Romanian Academy,
Bucharest, Romania}
\email{gmarines@math.uni-koeln.de}
\thanks{G.\ M.\ partially supported by DFG funded
project SFB TRR 191}


\begin{abstract}
We show that the Bergman kernel of a finite-volume quotient of a 
Hermitian manifold $\wi{X}$ with bounded geometry by a discrete group $\Gamma$ 
of its isometries is the same as the averaging over $\Gamma$ of the Bergman kernel on $\wi{X}$. 
We then use these results when $\wi{X}$ is a Hermitian symmetric space to show 
that a large class of relative Poincar\'e series does not vanish. 
This extends the results of Borthwick-Paul-Uribe and Barron (formerly Foth) 
to the case of general locally symmetric spaces of finite volume.
\end{abstract}

\maketitle

\tableofcontents

\section*{Introduction}

Let $(\wi{X},\wi{J},g^{T\wi{X}})$ be a complete
Kähler manifold of complex dimension $n$, equipped with a
holomorphic Hermitian line bundle $(\wi{L},h^{\wi{L}})$
whose Chern curvature satisfies
\begin{equation}\label{preq}
\frac{\sqrt{-1}}{2\pi}R^{\widetilde{L}}(\cdot,\wi{J}\cdot)=g^{T\wi{X}}\,,
\end{equation}
and let also $\Gamma$ be a discrete group acting properly discontinuously
and effectively on $(\wi{X},\wi{J},g^{T\wi{X}})$
by biholomorphic isometries lifting
to a holomorphic Hermitian action on $(\wi{L},h^{\wi{L}})$,
so that
the quotient space $X:=\wi{X}/\Gamma$ has a natural
structure of a Kähler orbifold $(X,J,g^{TX})$, equipped with a positive
orbifold line bundle $(L,h^L)$ whose Chern curvature satisfies \eqref{preq}
for $g^{TX}$ in the same way.
For any $p\in\N^*$, we write $\wi{L}^p:=\wi{L}^{\otimes p}$
for the $p$-th tensor power of $\wi{L}$ over $\wi{X}$,
and for each $x,\,y\in\wi{X}$, we consider the associated \emph{Bergman kernel} $\wi{P}_p(x,y)\in\wi{L}^p_x\otimes(\wi{L}^p_y)^*$
defined as the Schwartz kernel with respect to the Riemannian volume form
of $(\wi{X},g^{T\wi{X}})$ of the orthogonal projection operator
$\wi{P}_p:L^2(\wi{X},\wi{L}^p)\to H^0_{(2)}(\wi{X},\wi{L}^p)$ from the
space of square integrable sections to the subspace $H^0_{(2)}(\wi{X},\wi{L}^p)$
of square-integrable holomorphic sections of $\wi{L}^p$. Similarly,
we consider for each $x,\,y\in X$ the associated Bergman kernel $P_p(x,y)\in L^p_x\otimes (L^p_y)^*$
defined analogously as the Schwartz kernel of
the orthogonal projection operator
$P_p:L^2(X,L^p)\to H^0_{(2)}(X,L^p)$ on square-integrable holomorphic sections of $L^p$.
The first main result of this paper in then the following.

\begin{theorem}\label{th-covintro}
Suppose that $(\wi{X},g^{T\wi{X}})$ and $(\wi{L},h^{\wi{L}})$
have bounded geometry
in the sense of Definition \ref{bndedgeomdef}.
Then if the volume of $X:=\widetilde{X}/\Gamma$ is finite,
there exists $p_0\in\N^*$ such that for any
$p\geq p_0$, we have
\begin{equation}\label{p=sumpintro}
\sum_{g\in\Gamma}(g^{-1},1).\wi{P}_p(g.\wi{x},\wi{y})
=P_p(\pi(\wi{x}),\pi(\wi{y}))\;,
\end{equation}
where for any compact set of $K\subset\wi{X}$,
the convergence is uniform and absolute in $\wi{x}, \wi{y}\in K\times\wi{X}\cup\wi{X}\times K$, and 
where $(g^{-1},1)\in\Gamma\times\Gamma$ acts on $\wi{L}^p\boxtimes (\wi{L}^p)^*$
over $\wi{X}\times\wi{X}$ for all $g\in\Gamma$.
\end{theorem}

Theorem \ref{th-covintro} is a special case of Theorem \ref{th-cov},
which holds more generally when $g^{T\wi{X}}$ does not necessarily satisfies \eqref{preq}
and when $\wi{L}^p$ is tensored by a general holomorphic Hermitian vector bundle $(\wi{E},h^{\wi{E}})$ for each $p\in\N$. In the
particular
case when $X=\widetilde{X}/\Gamma$ is smooth and compact,
Theorem \ref{th-covintro} has been established in \cite[Theorem 2]{MM15}
under the more general hypotheses of Theorem \ref{th-cov}
and in the general symplectic case, and also by Lu and Zelditch
in \cite[Theorem 2]{LZ16}.
The proof of Theorem \ref{th-covintro} given in Section \ref{s3.4} is based on
the exponential decay of the Bergman kernel of $\wi{X}$ away from the diagonal
established in \cite[Theorem 1]{MM15}, which we recall in Theorem
\ref{thm:3.2new23}, as well as a comparison of the Bergman kernel with the
heat kernel of the associated \emph{Kodaira Laplacian} as $p\to+\infty$.

The main interest of Theorem \ref{th-covintro} lies in the fact that
for any fixed $\wi{y}\in\wi{X}$, formula \eqref{p=sumpintro}
provides an explicit construction of a $\Gamma$-invariant holomorphic
section of $\wi{L}^p$ by averaging over $\Gamma$.
As emphasized for instance in \cite[II.\,Chap.\,7]{Kol95},
in the particular case
when $\wi{X}$ is a Hermitian symmetric space,
its Bergman kernel admits a completely explicit formula, and
formula \eqref{p=sump} coincides with
the classical \emph{Poincar\'e series} of the theory of automorphic forms.
In Section \ref{s5.4}, we show how fundamental properties of the Bergman kernels
of compact orbifolds established by Dai, Liu and Ma in \cite{DLM06}
and extended to the non-compact case by Ma and Marinescu in \cite{MM04a}
translate via Theorem \ref{th-covintro}
into the non-vanishing of the classical Poincar\'e series respectively associated with
$SL_2(\R)^n$ in Proposition \ref{Ex1}, with $\Sp_{2n}(\R)$ in
Proposition \ref{Ex3}
and with $SU(n,1)$ in Proposition \ref{Ex4},
as soon as their \emph{weight} $p\in\N$ is large enough.


More generally, Section \ref{s5.2} explains how to apply Theorem 5.1 to the method 
of \emph{relative Poincar\'e series} in the context of the theory of automorphic forms. 
The goal is to produce holomorphic sections of $L^p$ that are 
invariant under the action of the group $\Gamma$ from holomorphic sections 
that are invariant under the action of a subgroup $\Gamma_0$ of $\Gamma$. 
This is achieved by averaging over the set $\Gamma_0\backslash\Gamma$ 
of cosets $[g]:=\{g_0g\in\Gamma:g_0\in\Gamma_0\}$ for all $g\in\Gamma$.

Specifically, using the asymptotic expansion of isotropic states
established in \cite[Theorem 5.3]{Ioo18b},
we show in Theorem \ref{t5.6} that
the analogue of formula \eqref{p=sumpintro} holds for such averages
when one replaces the role of a point $\wi{y}\in\wi{X}$
by a $\Gamma_0$-invariant compact submanifold of $X:=\wi{X}/\Gamma$,
and that the $\Gamma_0$-invariant holomorphic
section thus constructed does not vanish as soon as $p\in\N$
is large enough.
This result is valid under the fundamental geometric assumption 
that the underlying $\Gamma_0$-invariant submanifold satisfies the \emph{Bohr-Sommerfeld condition}, 
a notion originating from the theory of geometric quantization, 
which is further elaborated upon in Section \ref{BSsec}.

In the special
case of the relative Poincar\'e series associated with $SL_2(\R)$,
where the discrete subgroup $\Gamma\subset SL_2(\R)$ acts on
the hyperbolic
plane $\wi{X}=\mathbb{H}$ and $\Gamma_0\subset\Gamma$ is the subgroup generated by an \emph{hyperbolic element}
$g_0\in\Gamma$, the underlying $\Gamma_0$-invariant submanifold coincides with the closed geodesic
of $X:=\mathbb{H}/\Gamma$ generated by $g_0$, which automatically satisfies the Bohr-Sommerfeld
condition. In the case when $X:=\mathbb{H}/\Gamma$
is smooth and compact, the non-vanishing of
the associated relative Poincaré series is then a result of
Borthwich, Paul and Uribe in \cite[Theorem 4.11]{BPU95}, and the general case
when $X:=\mathbb{H}/\Gamma$ is only assumed to have finite volume is
the result of \cite[Theorem 6.7,\,(6.18)]{Ioo18b}, which we state and reprove in
Theorem \ref{Ex2} as a consequence of our set-up.

On the other hand, in the case the relative Poincar\'e series associated with $SU(n,1)$,
where the discrete subgroup $\Gamma\subset SU(n,1)$ acts on
the unit ball $\wi{X}=B_n$ in $\C^n$ and
$\Gamma_0\subset\Gamma$ is the subgroup generated by an \emph{loxodromic element}
$g_0\in\Gamma$, the underlying $\Gamma_0$-invariant submanifold again
coincides with the closed geodesic
of $X:=B_n/\Gamma$ generated by $g_0$, and Barron showed in
 \cite[\S\,2.2]{Bar18} that this geodesic
satisfies the Bohr-Sommerfeld condition if all eigenvalues of $g_0\in\Gamma$ are real
and if the endpoints $x,\,y\in\partial B_n$ of the geodesic $\wi{\gamma}\subset  B_n$ generated by $g_0$
belong to $\R^n\subset\C^n$. We then obtain
the following result as a consequence of Theorem \ref{th-covintro},
which shows once again that these relative Poincaré series do not
vanish as soon as their weight $p\in\N^*$ is large enough.

\begin{theorem}\label{Ex5intro}
Let $\Gamma\subset\SU(n,1)$ be a discrete subgroup acting effectively
and satisfying
$\text{Vol}(B_n/\Gamma)<+\infty$,
let $$g_0=
\begin{pmatrix}
A & b \\
c^{T} & d
\end{pmatrix}\in\Gamma,$$
with $A\in M_n(\C)$, $b,\,c\in\C^n$ et $d\in\C$,
be a loxodromic element with real eigenvalues,
and let $\Gamma_0\subset\Gamma$ be the subgroup
generated by $g_0$. Let $\wi\gamma:\R\rightarrow B_n$
be the geodesic generated by $g_0$, and let
$x,\,y\in\partial B_n$ be the two points of the boundary
of $B_n$ joined by $\wi\gamma$, and assume that $x,\,y\in\R^n$.
Then there exists $p_0\in\N^*$ such that  for
all $p\geq p_0$, the convergent series
\begin{equation}\label{RPSloxintro}
\sum_{[g]\in\Gamma_0\backslash\Gamma}
(c.z+d)^{-2(n+1)p}\frac{1}
{(\langle g.z,x\rangle\langle g.z,y\rangle)^{(n+1)p}}
\end{equation}
does not vanish identically in
$z\in B_n$, where
$\langle\cdot,\cdot\rangle$ is the canonical Hermitian product of $\C^n$
and $c.z:=\sum_{j=1}^nc_jz_j$.
\end{theorem}

Theorem \ref{Ex5intro} is a consequence of Theorem \ref{Ex5}, and the explicit
formula \eqref{RPSloxintro} for the relative Poincar\'e series
has been computed by Foth and Katok in \cite[\S\,6.3]{FK01}, who also study their applications
to the arithmetic theory of automorphic forms.
In the particular case when $X:=B_n/\Gamma$ is smooth and compact,
Theorem \ref{Ex5intro} is the result of Barron in \cite[Theorem 3.3]{Bar18}.

Finlly, if $n=2$ and for a loxodromic element $g_0\in\Gamma$
with real eigenvalues in $\SU(2,1)$,
Barron considers in \cite[\S\,3,\,\S\,4]{Fot02} relative Poincar\'e series 
associated with some remarkable Lagrangian tori inside $X:=B_2/\Gamma$
satisfying the Bohr-Sommerfeld condition,
and we show in Theorem \ref{Ex6}
that these relative Poincar\'e series do not vanish as soon as
their weight $p\in\N$
is large enough.
In the particular case when $X:=B_2/\Gamma$ is smooth and compact,
this is the result of Barron in \cite[\S\,4]{Fot02}.
Let us also note that for general $n\in\N^*$ and in the case $X:=B_n/\Gamma$
smooth and compact,
Alluhaibi and Barron studied in \cite{AB17} the case of relative
Poincar\'e series associated with submanifolds of $B_n$
that are not necessarily isotropic.

This paper is organized as follows. In Section \ref{setting},
we recall the basic set-up of Bergman kernels on manifolds with
bounded geometry following for instance \cite[\S\,6.1]{MM06},
then recall the exponential decay of the Bergman kernel away from
the diagonal established in \cite[Theorem 1]{MM15}, which
forms the
basis of Theorem \ref{th-covintro}. In Section \ref{s3.4},
we give the proof of Theorem \ref{th-covintro}, while in
Section \ref{BSsec}, we introduce the notion of a
Bohr-Sommerfeld submanifold and recall in Theorem
\ref{proprepgal} the asymptotic expansion of the
associated isotropic
states established in \cite[Theorem\,5.3]{Ioo18b},
which forms the basis of our applications to relative
Poincaré series in the next sections. In Section \ref{s5.2},
we define relative Poincaré series on general
covering manifolds and establish their existence in
Theorem \ref{t5.6} based on our main Theorem \ref{th-covintro}. Finally, in Section \ref{s5.3} we specialize
to the case of Hermitian symmetric spaces, showing that
one recovers in this way the classical Poincaré series
and their relative version, while in Section \ref{s5.4},
we apply the results of the previous sections to establish
the non-vanishing of a large class of relative Poincaré
series as soon as their weight is large enough, giving
among many others the proof of Theorem \ref{Ex5intro}.

\section{Setting}
\label{setting}

We consider a complex manifold $(X,J)$ with complex structure $J$,
and complex dimension $n$.
Let $g^{TX}$ be a $J$-invariant Riemannian metric on $X$, 
also called \emph{Hermitian metric}, and write $d^X(\cdot,\cdot)$ for the
associated Riemannian distance.
Let $\Theta$ be the real $(1, 1)$-form defined 
by		
\begin{align}\label{2.18}		
\Theta(X, Y)=g^{TX}(JX, Y).		
\end{align}		
The Riemannian volume form of $g^{TX}$ is given by
$$dv_X=\Theta^n/n!\,.$$			
Let $(E,h^E)$ be a holomorphic Hermitian vector bundle on $X$, and write $R^E\in\Omega^2(X,\textup{End}(E))$ for its Chern curvature.
On the space of smooth sections with compact support
$\cC^{\infty}_0(X,E)$, 
we introduce
the following $L^2$-inner product associated with $h^E$
and $dv_X$,
defined for any $s_1,\,s_2\in\cC^{\infty}_0(X,E)$ by
\begin{equation}\label{lm2.0}
\big\langle  s_1,s_2 \big\rangle	
	:=\int_X\big\langle  s_1(x),s_2(x)\big\rangle_{E}\,dv_X(x)\,,
\end{equation}
%
%
%
where $\langle\cdot,\cdot\rangle_{E}$ is the Hermitian
product on $E$ induced by $h^E$, and we write $\|\cdot\|_{L^2}$ for the associated $L^2$-norm.
The completion of $\cC^{\infty}_0(X,E)$
with respect to \eqref{lm2.0} is denoted by $L^{2}(X,E)$,
and we write
\begin{equation}\label{lm2.02a}		
H^{0}_{(2)}(X,E):=		
\big\{s\in L^{2}(X,E) : \text{$s$ is holomorphic}\big\}\,.		
\end{equation}	
for the space of holomorphic $L^{2}$ sections of $E$.

We deduce from the Cauchy estimates for holomorphic functions that
for every compact set
$K\subset X$ there exists $C_K>0$ such that for all 
$s\in H^{0}_{(2)}(X,E)$,
\begin{equation}\label{b1.1}
\sup_{x\in K}|s(x)|\leqslant C_K\|s\|_{L^2}\,.
\end{equation}
This implies that $H^{0}_{(2)}(X,E)$ is a closed subspace of
$L^2(X,E)$. Moreover,
$H^{0}_{(2)}(X,E)$ is separable (cf.\ \cite[p.\,60]{Weil:58}).

Let now $T^{(1,0)}X$ be the holomorphic tangent bundle on $X$,
$T^{(0,1)}X$ the conjugate of $T^{(1,0)}X$
and $T^{*(0,1)}X$ the dual bundle of $T^{(0,1)}X$.
We denote by $\Lambda^q(T^{*(0,1)}X)$ the bundle of $(0,q)$-forms
on $X$ and by
$\Omega^{0,q}(X,E)$ the space of sections of the bundle 
$\Lambda^q(T^{*(0,1)}X)\otimes E$ over $X$.
We denote by $\Omega_0^{0,{\bullet}}(X,E)$		
the subspace of $\Omega^{0,{\bullet}}(X,E)$ consisting of elements		
with compact support.		
The Dolbeault operator acting on sections of the holomorphic vector
bundle $E$ gives rise to the Dolbeault complex
\(\big(\Omega^{0,\bullet}(X,E), \overline{\partial}^{E}\big)\,.\)
We denote by $\overline{\partial}^{E,*}$ the formal adjoint 
of $\overline{\partial}^{E}$
with respect to the $L^2$-inner product \eqref{lm2.0}.
Set
\begin{align}\label{lm2.1}
D = \sqrt{2}\big(\, \overline{\partial}^{E}
+ \,\overline{\partial}^{E,*}\,\big)
\,.
\end{align}
Then its square
$D^2 = 2\big(\, \overline{\partial}^{E}\overline{\partial}^{E,*}
+ \,\overline{\partial}^{E,*}\overline{\partial}^{E}\,\big)$
is twice the \emph{Kodaira Laplacian} of $(E,h^E)$, which
preserves the $\Z$-grading on $\Omega^{0,\bullet}(X,E)$,
and whose kernel restricted to $\Omega^{0,0}(X,E)$ coincides
with the holomorphic sections of $E$.

Let us also consider a holomorphic Hermitian line bundle $(L,h^L)$ over $X$.
We will make the following assumptions.

\begin{definition}\label{bndedgeomdef}
We say that $(X,J,g^{TX})$, $(E,h^E)$ and $(L,h^L)$
have \emph{bounded geometry} if $(X,g^{TX})$ is complete with positive injectivity radius,
if the derivatives of any 
order of $R^E$, $R^L$, $J$ and $g^{TX}$ are uniformly bounded on $X$ 
in the norm induced by $g^{TX}$, $h^{E}$ and $h^L$.
\end{definition}

We will also make the assumption, extending \eqref{preq},
that there exists
$\varepsilon>0$ such that
\begin{equation}\label{RL>eps}
\sqrt{-1}R^L(\cdot,J\cdot)> \varepsilon g^{TX}\;.
\end{equation}
One fundamental example of a Hermitian manifold $(X,J,g^{TX})$ with bounded geometry is
a Galois 
covering of a compact K\"{a}hler manifold $M$,
so that $(X,J,g^{TX})$ satisfies the assumptions of
Definition \ref{bndedgeomdef} by
taking the line bundle $(L,h^L)$ to be the pull-back of a positive holomorphic line bundle on $M$, so that it satisfies the assumption
\eqref{RL>eps}.

For any $p\in\N$, write $L^p:=L^{\otimes p}$ for the $p$-th tensor power of $L$, and equip the tensor product $L^p\otimes E$
with the Hermitian metric $h^{L^p\otimes E}$ induced by $h^L$ and $h^E$.

\begin{definition} \label{almt2.1b}
For any $p\in\N$, the \emph{Bergman projection} is the orthogonal projection
\[
P_p:L^{2}(X,L^p\otimes E)\to H^{0}_{(2)}(X,L^p\otimes E)\,.
\]
\end{definition}
By \eqref{b1.1}, for a fixed $x\in X$
the evaluation functional $s\mapsto s(x)$ on
$H^{0}_{(2)}(X,L^p\otimes E)$
is continuous.
By the Riesz representation theorem there exists
$P_p(x,\cdot)\in L^2(X,(L^p\otimes E)_x\otimes (L^p\otimes E)^{*})$
such that
\begin{equation}\label{lm2.01a}
s(x)=\int_{X}P_p(x,x') s(x') dv_{X}(x')\,,
\quad\text{for all $s \in H^{0}_{(2)}(X,L^p\otimes E)$\,.}
\end{equation}
\begin{definition} \label{almt2.1c}
For any $p\in\N$, the section $P_p(\cdot,\cdot)$ of
$L^p\otimes E\boxtimes (L^p\otimes E)^{*}$ over
$X\times X$ is called the \emph{Bergman kernel} of
$H^{0}_{(2)}(X,L^p\otimes E)$.
\end{definition}
For any $p\in\N$, let $d_p:= \dim H^{0}_{(2)}(X,L^p\otimes E)\in\N\cup\{\infty\}$, and 
let $\{s_i\}_{i=1}^{d_p}$ be any orthonormal basis of
$H^{0}_{(2)}(X,L^p\otimes E)$ with respect to the inner
product \eqref{lm2.0}. Using the estimate \eqref{b1.1} we can show that
\begin{equation} \label{bk2.4}
P_p(x,x')= \sum_{i=1}^{d_p} s_i (x) \otimes s_i(x')^*
\in (L^p\otimes E)_x\otimes (L^p\otimes E)_{x'}^*\,,
\end{equation}
where the right-hand side converges on every compact subset of $X$
together with all its derivatives (see e.\,g.\ \cite[p.\,62]{Weil:58}).
Thus $P_p(\cdot,\cdot)\in \cC^{\infty}(X\times X,L^p\otimes E\boxtimes (L^p\otimes E)^{*})$.
It follows from \eqref{bk2.4} that
\begin{equation}\label{lm2.01}
(P_ps)(x)=\int_{X}P_p(x,x') s(x') dv_{X}(x')\,,
\quad\text{for all $s\in L^2(X,L^p\otimes E)$\,.}
\end{equation}
that is, $P_p(\cdot,\cdot)$ is the integral kernel (hence Schwartz kernel)
of the Bergman projection $P_p$.

We end this section by recalling
the following fundamental result about the 
exponential decay of the Bergman kernel.
\begin{theorem}[{\cite[Theorem 1]{MM15}}]\label{thm:3.2new23}
Suppose that $(\wi{X},g^{T\wi{X}})$, $(\wi{L},h^{\wi{L}})$ and $(\wi{E},h^{\wi{E}})$
have bounded geometry in the sense of Definition \ref{bndedgeomdef}
and that assumption \eqref{RL>eps} is satisfied.
Then there exist $\boldsymbol{c} >0$, $\boldsymbol{p}_{0}>0$, 
which can be determined explicitly
from the geometric data 
such that  for any $k\in \N$, there exists $C_k>0$ such that 
for any $p\geqslant \boldsymbol{p}_{0}$, $x,x'\in X$, we have 
\begin{equation}\label{eq:0.7}
\left| P_p(x,x')\right|_{\cC^k} \leqslant C_k \, p^{n+\frac{k}{2}}
\, \exp\!\left(- \boldsymbol{c} \,\sqrt{p}\, d(x,x')\right).
\end{equation}
\end{theorem}

\section{Covering manifolds}\label{s3.4}

Let $(\wi{X},g^{T\wi{X}})$ be a complete
Riemannian manifold of dimension $m$,
and let $\Gamma$ be a discrete group
acting properly discontinuously and effectively
on $(\wi{X},g^{T\wi{X}})$ by isometries.
For any $r>0$ and any subset
$F\subset\wi{X}$, set
\begin{equation}\label{NF}
N_F(r)=\sup\limits_{x\in F,\,y\in\wi{X}}\#\left\{g\in\Gamma~:~
g.x\in B^{\wi{X}}(y,r)\right\}\;,
\end{equation}
where $B^{\wi{X}}(y,r)$ is the geodesic ball of center $y\in\wi{X}$
and radius $r$.

\begin{lemma}\label{t3.2}
Suppose that there exists $K>0$ such that the Ricci curvature
$\textup{Ric}$ of $(\wi{X},g^{T\wi{X}})$ satisfies
$\textup{Ric}\geq -(m-1)K^2g^{T\wi{X}}$.
Then for any compact subset $F\subset \wi{X}$,
there exists $C_F>0$ such
that for all $r>0$, the following estimate holds,
\begin{equation}\label{NF<exp}
N_F(r)\leq C_F\,e^{(m-1)Kr}\;.
\end{equation}
\end{lemma}
\begin{proof}
As the action of
$\Gamma$ on $\wi{X}$ is properly discontinuous, for any $x\in X$, its stabilizer $\Gamma_x\subset\Gamma$
is a finite group, and there exists $\epsilon_x>0$ such that
for any $g\in\Gamma$, we have
\begin{equation}\label{BgB=0}
B^{\wi{X}}(x,4\epsilon_x)\cap g.B^{\wi{X}}(x,4\epsilon_x)=\emptyset
\quad\text{if and only if}\quad g\notin\Gamma_x\,.
\end{equation}
In particular, for any $y\in B(x,2\epsilon_x)$ and $g\notin\Gamma_x$, we have
$B^{\wi{X}}(y,2\epsilon_x)\cap g.B^{\wi{X}}(y,2\epsilon_x)=\emptyset$, which shows that
$\Gamma_y\subset\Gamma_x$. As $F$ is compact, there exists a finite set
$\{x_1,\dots,x_k\}\subset F$ with $k\in\N$ such that
$F\subset\cup_{j=1}^k B^{\wi{X}}(x_j,\epsilon_{x_j})$, and
for any $y\in\wi{X}$ and $1\leq j\leq k$, set
\begin{equation}\label{Njx}
N_{j,y}(r)=\sup\limits_{x\in B^{\wi{X}}(x_j,\epsilon_{x_j})}
\#\left\{g\in\Gamma~:~
g.x\in B^{\wi{X}}(y,r)\right\}\;.
\end{equation}
Then by \eqref{NF}, we have
\begin{equation}
N_F(r)\leq \sup_{y\in\wi{X}}\,\sup_{1\leq j \leq k}N_{j,y}(r).
\end{equation}
Thus by \eqref{Njx}, it suffices to show \eqref{NF<exp} for
$N_{j,y}(r)$, with a constant that does not depend on $y\in\wi{X}$.

Let us first consider the case of $j\in\N$ such that
$\Gamma_{x_j}=\{1\}$. Then
as $\Gamma$ acts by
isometries and using \eqref{BgB=0}, for any $r>0$ we have
\begin{equation}\label{Nvol<vol}
N_{j,y}(r)~\text{Vol}\,B^{\wi{X}}(x_j,\epsilon_{x_j})\leq
\text{Vol}\,B^{\wi{X}}(y,r+2\epsilon_{x_j})\;.
\end{equation}
Now under our assumption on bounded geometry, there exists $K>0$ such that 
the Ricci curvature of $(\wi{X}, g^{T\wi{X}})$ is bounded below by $-(2m-1) K^2 g^{T\wi{X}}$. By Bishop's inequality, this implies that the volume of $B^{\wi{X}}(x,r)\subset\wi{X}$
is smaller or equal to the volume
of a geodesic ball of radius $r>0$ in the space of constant
curvature $-K$ (see for example \cite[Lemma\,7.1.3]{Pet16}).
Then by a classical estimate
for the volume of large balls in the space of constant curvature
$-K$, which can be found for example in \cite[p.\,3]{Mil68},
there exists a universal constant $C_{m,K}>0$, depending only on
$K$ and on the dimension $m$ of $\wi{X}$, such that for any $x\in\wi{X}$ and $r>0$,
\begin{equation}\label{Bishop}
\text{Vol}\,B^{\wi{X}}(x,r)\leq C_{m,K}\,e^{(m-1)Kr}\;.
\end{equation}
Then for any $r>0$ we get,
\begin{equation}
N_{j,y}(r)\leq\frac{C_{m,K}\,e^{(m-1)K2\epsilon_{x_j}}}
{\text{Vol}\,B^{\wi{X}}(x_j,\epsilon_{x_j})}\,e^{(m-1)Kr}\;,
\end{equation}
and the right hand side does not depend on $y\in\wi{X}$.

For the general case, first recall that we chose the cover
of $F$ by geodesic balls as above such that for any $1\leq j\leq k$,
the stabilizer $\Gamma_{x_j}\subset\Gamma$ of $x_j\in F$
satisfies \eqref{BgB=0}. Then as $\Gamma$ acts by isometries,
we know that $\Gamma_{x_j}$ preserves the geodesic
ball $B^{\wi{X}}(x_j,\epsilon_{x_j})$, and using \eqref{BgB=0},
we deduce as above that
\begin{equation}
\frac{N_{j,y}(r)}{\#\Gamma_{x_j}}\,
\text{Vol}\,B^{\wi{X}}(x_j,\epsilon_{x_j})\leq\text{Vol}\,
B^{\wi{X}}(y,r+2\epsilon_{x_j})\;,
\end{equation}
so that
\begin{equation}
N_{j,y}(r)\leq \frac{C_{m,K}\,e^{(m-1)K2\epsilon_{x_j}}\,\#\Gamma_{x_j}}
{\text{Vol}\,B^{\wi{X}}(x_j,\epsilon_{x_j})}\,e^{(m-1)Kr}\;,
\end{equation}
and the right hand side does not depend on $y\in\wi{X}$.
This finishes the proof of \eqref{NF<exp}.
\end{proof}

Let $(\wi{X},\wi{J},g^{T\wi{X}})$ be a complete
Hermitian manifold of complex dimension $n$,
and let $\Gamma$ be a discrete group acting properly discontinuously
and effectively on $(\wi{X},\wi{J},g^{T\wi{X}})$
by biholomorphic isometries.
Then the quotient space $X:=\wi{X}/\Gamma$ has a natural
structure of a Hermitian orbifold $(X,J,g^{TX})$. Furthermore,
 any Cauchy
sequence for the orbifold Riemannian distance $d^X$ of
$(X,g^{TX})$ can be
lifted to a Cauchy sequence for the Riemannian distance
$d^{\wi{X}}$ of
$(\wi{X},g^{T\wi{X}})$, so that the completeness of
$(\wi{X},d^{\wi{X}})$
implies the completeness of $(X,d^X)$. We write
$\pi:\wi{X}\rightarrow X$ for the canonical projection.

Assume that the action of $\Gamma$ lifts
to an action on holomorphic Hermitian vector bundles
$(\wi{L},h^{\wi{L}})$ and $(\wi{E},h^{\wi{E}})$, with
$(\wi{L},h^{\wi{L}})$ a uniformly positive line bundle, so that
its Chern curvature $R^{\widetilde{L}}$ satisfies \eqref{RL>eps}.
Then as $\Gamma$ acts effectively on $\wi{X}$, the
quotient bundles $(L,h^L)$ and $(E,h^E)$ are proper orbifold
bundles on $X$.
For any $p\in\N^*$,
write
$\wi{D}_p,\,D_p$ for the Dirac operators of
$\wi{L}^p\otimes \wi{E}$, $L^p\otimes E$
as in \eqref{lm2.1},
and $\wi{P}_p(\cdot,\cdot)$, $P_p(\cdot,\cdot)$
for the associated Bergman kernels with respect to the
Riemannian volume forms $dv_{\wi{X}}$, $dv_{X}$ of
$(\wi{X},g^{T\wi{X}})$, $(X,g^{TX})$ as in \eqref{lm2.01a}.

\begin{theorem}\label{th-cov}
Suppose that $(\wi{X},g^{T\wi{X}})$, $(\wi{L},h^{\wi{L}})$ and $(\wi{E},h^{\wi{E}})$
have bounded geometry in the sense of Definition \ref{bndedgeomdef}
and that \eqref{RL>eps} holds.
Then if the volume of $X:=\widetilde{X}/\Gamma$ is finite,
there exists $p_1\in\N^*$ such that for any
$p\geq p_1$, we have
\begin{equation}\label{p=sump}
\sum_{g\in\Gamma}(g^{-1},1).\wi{P}_p(g.\wi{x},\wi{y})
=P_p(\pi(\wi{x}),\pi(\wi{y}))\;,
\end{equation}
where for any compact set of $K\subset\wi{X}$,
the convergence is uniform and absolute in $\wi{x}, \wi{y}\in K\times\wi{X}\cup\wi{X}\times K$.
\end{theorem}
\begin{proof}
First note that if $\Gamma$ acts freely on $\wi{X}$, then the quotient $E=\wi{E}/\Gamma$ identifies
$v\in\wi{E}_{\wi{x}}$ with $g.v\in\wi{E}_{g.\wi{x}}$ for any $x\in\wi{X}$ and $g\in\Gamma$,
so that we can omit the action of $(g^{-1},1)$ on 
$(\wi{L}^p\otimes \wi{E})_{g.\wi{x}}\otimes(\wi{L}^p\otimes \wi{E})_{\wi{y}}$ in \eqref{p=sump}. Recall that the Dirac
operators $D_p$, $\wi{D}_p$ defined in \eqref{lm2.1} acts on the
smooth sections of
\begin{equation}
    E_p:=\Lambda^\bullet(T^{*(0,1)}X)\otimes L^p\otimes E\quad\text{and}\quad \wi{E}_p:=\Lambda^\bullet(T^{*(0,1)}\wi{X})\otimes \wi{L}^p\otimes \wi{E}\,.
\end{equation}
In general,
using \cite[Lemma\,1, Remark\,1]{MM15} as in the proof of
\cite[Theorem 6.1.1]{MM07},
the assumption on the boundedness of
$R^{\wi{L}},\,R^{\wi{E}},\,\wi{J},\,g^{T\wi{X}}$ implies that
\begin{equation}
\begin{split}
\textup{Spec}(D_p^2)&\subset\{0\}\cup[2p\mu_0-C_1,+\infty[\;,\\
\textup{Spec}(\wi{D}_p^2)&\subset\{0\}\cup[2p\mu_0-C_1,+\infty[\;,
\end{split}
\end{equation}
with
\begin{align}\label{mu0}
\mu_{0}=\inf_{w\in T^{(1, 0)}_{x}X,\ x\in X}
\frac{R^{L}(w, \overline{w})}{|w|^{2}_{g^{TX}}}>\varepsilon
\end{align}
for some $C_{1}>0$ and $\varepsilon>0$ as in \eqref{RL>eps}, and that
for $p\in\N^*$ large enough, we have
\begin{equation}
\begin{split}
\textup{Ker}(D_p^2)&=H^0_{(2)}(X,L^p\otimes E)\;,\\
\textup{Ker}(\wi{D}_p^2)&=H^0_{(2)}(\wi{X},\wi{L}^p\otimes \wi{E})\;.
\end{split}
\end{equation}
Let us first show that for any $p\in\N$ and $u>0$,
the following formula as in \cite[(3.18)]{MM15} holds,
\begin{equation}\label{e=sume}
\exp(-u D_p^2)(\pi(\wi{x}),\pi(\wi{y}))=\sum_{g\in\Gamma}
(g^{-1},1).\exp(-u\wi{D}_p^2)(g.\wi{x},\wi{y})\;,
\end{equation}
with uniform convergence for $\wi{x}\in\wi{X}$ and $\wi{y}$ in any
compact set $F\subset\wi{X}$.

For any
$p\in\N$ fixed, by applying \cite[Theorem 4]{MM15} for
$L=\C$ the trivial line bundle, we get
$a>0$ such that for any $k\in\N$, $u_0>0$, there exists $C>0$ such that for any
$0<u<u_0$ and $\wi{x},\,\wi{y}\in\wi{X}$, we have
\begin{equation}\label{cvexpufl0}
\left|\exp\left(-u\wi{D}_{p}^2\right)(\wi{x},\wi{y})\right|_{\cC^k}
\leqslant C u^{-n-\frac{k}{2}}
\exp\left(-\frac{a}{u}d^{\wi{X}}(\wi{x},\wi{y})^2\right)\;.
\end{equation}
Thus by Lemma \ref{t3.2}, for any $k\in\N$, $u_0>0$ and
any compact set $F\subset\wi{X}$ , there exists
$C>0$ such that for all $0<u<u_0$, $\wi{x}\in\wi{X}$ and all $\wi{y}\in F$, we have
\begin{equation}\label{explicitunifcv}
\begin{split}
\sum_{g\in\Gamma}\,\left|\exp\left(-u\wi{D}_{p}^2\right)(g.\wi{x},\wi{y})\right|_{\cC^k}
&\leq C\sum_{g\in\Gamma}u^{-n-\frac{k}{2}}
\exp\left(-\frac{a}{u}d^{\wi{X}}(\wi{x},g^{-1}.F)^2\right)\\
&\leq C\,\sum_{r=0}^\infty\,u^{-n-\frac{k}{2}}\,N_F(r+1)\,e^{-\frac{a}{u}\,r^2}<+\infty\,.
\end{split}
\end{equation}
This shows that for any compact set $F\subset\wi{X}$ and any compact
interval $I\subset\,]0,+\infty[$, the right hand side
of \eqref{e=sume} and all its derivatives converge uniformly in
$\wi{x}\in\wi{X}$, $\wi{y}\in F$ and $u\in\,I$.
Identifying an open dense set $U\subset X$
with an open fundamental domain $\wi{U}$ on $\Gamma$ in $\wi{X}$,
this implies in particular that
for any $s\in\cC^{\infty}_0(X,E_p)$,
by considering the compact subset $F:=\overline{\pi^{-1}(\supp\,s)\cap\wi{U}}$, the sum
\begin{equation}\label{sum=eL2}
\sum_{g\in\Gamma}\big(e^{-u\wi{D}^2_p}s\big)(g.x)
\end{equation}
converges uniformly in $x\in X$ and $u\in\,I$ as well as all its derivatives,
where we wrote
\begin{equation}\label{eL2}
\big(e^{-u\wi{D}^2_p}s\big)(g.x)
:=\int_{\wi{U}}
(g^{-1},1).\exp(-u\wi{D}_p^2)(g.\wi{x},\wi{y})s(\wi{y})dv_{\wi{X}}(\wi{y})\,,
\end{equation}
with $\wi{x}\in\overline{\wi{U}}$ such that $\pi(\wi{x})=x$.
As the sum \eqref{sum=eL2} does not depend on the choice of
fundamental domain $\wi{U}\subset\wi{X}$, it defines
a smooth section in $\cC^\infty(X,E_p)$.
By the defining property of
heat operators, for all $x\in X$ and $u>0$, we get
\begin{equation}\label{defpropheat}
\left(\frac{\partial^2}{\partial u^2}+D_p^2\right)
\sum_{g\in\Gamma}\big(e^{-u\wi{D}^2_p}s\big)(g.x)=0\;.
\end{equation}
Thus by uniqueness of the heat operator,
to show \eqref{e=sume}, it suffices to show that
\eqref{sum=eL2} belongs to $L^2(X,E_p)$
 and that
\begin{equation}\label{expfl0}
\int_X\big|\sum_{g\in\Gamma}\big(e^{-u\wi{D}^2_p}s\big)(g.x)
-s(x)\big|^2\,dv_X(x)\xrightarrow{~~u\rightarrow 0~~}0\,,
\end{equation}
for any $s\in\cC^{\infty}_0(X,E_p)$.
To show that \eqref{sum=eL2} belongs to $L^2(X,E_p)$,
note that for any $\wi{x}\in\wi{U}$ and considering the compact subset
$F:=\overline{\pi^{-1}(\supp\,s)\cap\wi{U}}$, we have
\begin{equation}\label{sumfin}
\begin{split}
\big|\sum_{g\in\Gamma}\big(e^{-u\wi{D}^2_p}s\big)&(g.x)
\big|\\
&\leq \sup\limits_{y\in F}\sum_{g\in\Gamma}\left|\exp\left(-u\wi{D}_{p}^2\right)(g.\wi{x},\wi{y})\right|
\left(\int_{F}\,|s(\wi{y})|\,dv_{\wi{X}}(\wi{y})\right)\,.
\end{split}
\end{equation}
Together with \eqref{explicitunifcv} and the assumption that $\textup{Vol}(X)<+\infty$, this implies that
\eqref{sum=eL2} belongs to $L^2(X,E_p)$. 

To establish \eqref{expfl0}, note that using a partition of unity, we can always
assume either that the support of $s\in\cC^{\infty}_0(X,E_p)$
is contained in the interior of $U\subset X$, or either that it is
contained in an orbifold chart. Assuming first that it is contained
in the interior of $U$, we get as above that
\begin{equation}\label{sumetends}
\begin{split}
&\int_X\big|\sum_{g\in\Gamma}\big(e^{-u\wi{D}^2_p}s\big)(g.x)
-s(x)\big|^2\,dv_X(x)\\
&\leq\text{Vol}(X)
\Big(\sup\limits_{x\in X}\sum_{g\in\Gamma,\,g\neq 1}
\big|\big(e^{-u\wi{D}^2_p}s\big)(g.x)
\big|+
\sup\limits_{x\in X}\big|\big(e^{-u\wi{D}^2_p}s\big)(x)-s(x)\big|
\Big)^2\,.\\
\end{split}
\end{equation}
As $F:=\overline{\pi^{-1}(\supp\,s)\cap\wi{U}}$ is contained in the interior of $\wi{U}$, there exists
$b>0$ such that for any $\wi{x}\in\wi{U}$ and $g\in\Gamma\backslash\{1\}$, we have $b<d^{\wi{X}}(\wi{x},g^{-1}.F)$. Then \eqref{explicitunifcv} shows that the
first term of the right hand side of \eqref{sumetends} is bounded by
$$C\,\sum_{r=1}^\infty\,u^{-n}\,N_F((r+1)b)\,e^{-\frac{a}{u}\,r^2 b^2},$$ 
which
tends to
$0$ as $u\rightarrow 0$ by Lemma \ref{t3.2}, while by uniform convergence on compact sets
of the heat equation on $\wi{X}$
to its initial solution as
$u\rightarrow 0$, the
second term also tends to $0$
as $u\rightarrow 0$ by \eqref{eL2}. This establishes \eqref{expfl0} in case
$\supp(s)\subset X$ is contained in the interior of $U$.

Assume now that the support of  $s\in\cC^{\infty}_0(X,E_p)$ is contained
in an orbifold chart
$B^{\wi{X}}(\wi{x_0},\epsilon)\rightarrow
B^X(x_0,\epsilon)=B^{\wi{X}}(\wi{x_0},\epsilon)/\Gamma_{x_0}$,
where $\Gamma_{x_0}\subset\Gamma$ is the finite
stabilizer of $\wi{x_0}\in\wi{X}$ and where
$B^{\wi{X}}(\wi{x_0},\epsilon)\cap g.B^{\wi{X}}(\wi{x_0},\epsilon)=\emptyset$
for all $g\notin\Gamma_{x_0}$ by \eqref{BgB=0}.
By definition \cite[(1.5)]{Ma05} of the orbifold integral
and by $\Gamma$-equivariance of the heat kernel on $\wi{X}$,
writing $\Gamma_{x_0}\backslash\Gamma$ for the set of cosets
\begin{equation}
[g]:=\{g_0g\in\Gamma:g_0\in\Gamma_{x_0}\}\,,
\end{equation}
for all $g\in\Gamma$ and writing $\wi{s}\in\cC^{\infty}_0(\wi{X},\wi{E}_p)$
for the $\Gamma_{x_0}$-invariant lift of
$s\in\cC^{\infty}_0(X,E_p)$ to
$B^{\wi{X}}(\wi{x_0},\epsilon)$,
for any $x\in B^X(x_0,\epsilon)$
and $\wi{x}\in B^{\wi{X}}(\wi{x_0},\epsilon)$ satisfying $\pi(\wi{x})=x$,
we get
\begin{equation}\label{orbisum}
\begin{split}
\sum_{g\in\Gamma}\big(e^{-u\wi{D}^2_p}s\big)(g.x)
&=\frac{1}{\#\Gamma_{x_0}}\int_{B^{\wi{X}}(\wi{x_0},\epsilon)}
\sum_{g\in\Gamma}
\exp(-u\wi{D}_p^2)(g.\wi{x},\wi{y})s(\wi{y})dv_{\wi{X}}(\wi{y})\\
&=\frac{1}{\#\Gamma_{x_0}}\int_{B^{\wi{X}}(\wi{x_0},\epsilon)} \sum_{[g]\in\Gamma_{x_0}
\backslash\Gamma}
\sum_{g_1\in [g]}\exp(-u\wi{D}_p^2)(g_1.\wi{x},\wi{y})\wi{s}(\wi{y})dv_{\wi{X}}(\wi{y})\\
&=: \sum_{[g]\in\Gamma_{x_0}
\backslash\Gamma}\big(e^{-u\wi{D}^2_p}\wi{s}\big)(g.\wi{x})\,,
\end{split}
\end{equation}
 Then the last line of
\eqref{orbisum} can be seen as the $\Gamma_{x_0}$-invariant lift
of the left hand side of \eqref{orbisum} to
$B^{\wi{X}}(\wi{x_0},\epsilon)$. By definition of the
orbifold integral and using
$\supp(s)\subset B^X(x_0,\epsilon)$, we then get
\begin{multline}\label{orbiL2norm}
\int_X\big|\sum_{g\in\Gamma}\big(e^{-u\wi{D}^2_p}s\big)(g.x)
-s(x)\big|^2\,dv_X(x)\\
=\int_{X\backslash B^X(x_0,\epsilon)}
\big|\sum_{g\in\Gamma}\big(e^{-u\wi{D}^2_p}s\big)(g.x)\big|^2
\,dv_X(x)\\
+\frac{1}{\#\Gamma_{x_0}}
\int_{B^{\wi{X}}(\wi{x_0},\epsilon)}\Big|
\sum_{[g]\in\Gamma_{x_0}
\backslash\Gamma}\big(e^{-u\wi{D}^2_p}\wi{s}\big)(g.\wi{x})
-\wi{s}(\wi{x})\Big|^2\,dv_{\wi{X}}(\wi{x})\,.
\end{multline}
We can then repeat the estimate \eqref{sumetends} with each term
of the right
hand side of \eqref{orbiL2norm} instead, and get \eqref{expfl0}
from \eqref{cvexpufl0} as above.
This concludes the proof of \eqref{e=sume}.

Now using \cite[Theorem 4]{MM15} again, we know that there exists $a>0$
such that for any $u_0>0$, $k\in\N$,
there exists $C>0$ such that for any $u\geq u_0,\,p\in\N^*$
and $\wi{x},\,\wi{y}\in\wi{X}$, we have
\begin{equation}\label{estsumheatfinal}
\left|\exp\left(-\frac{u}{p}\wi{D}_p^2\right)(\wi{x},\wi{y})\right|_{\cC^k}
\leqslant C p^{n+\frac{k}{2}}
\exp\left(\mu_0 u-\frac{ap}{u}d^{\wi{X}}(\wi{x},\wi{y})^2\right)\;.
\end{equation}
Then following \cite[(4.2.22)]{MM07}, for any $p\in\N^*,\,u>0$, we have
\begin{equation}\label{e-p=inte}
\begin{split}
\exp\left(-\frac{u}{p} D_p^2\right)-P_p
&=\int_u^{\infty}\frac{1}{p}D_p^2
\exp\left(-\frac{u_{1}}{p}D_p^2\right)du_{1},
\\
\exp\left(-\frac{u}{p}\wi{D}_p^2\right)-\wi{P}_p&=
\int_u^{\infty}\frac{1}{p}
\wi{D}_p^2\exp\left(-\frac{u_{1}}{p}\wi{D}_p^2\right)du_{1}.
\end{split}
\end{equation}
Now by Theorem \ref{thm:3.2new23} and Lemma \ref{t3.2}, we know that
$\sum_{g\in\Gamma}\,(g^{-1},1).\wi{P}_p(g.\wi{x},\wi{y})$
is absolutely convergent with respect in $\cC^k$-norm of $\wi{X}\times K$ for any compact subset $K\subset\wi{X}$,
uniformly in $p>\max\{p_0,(m-1)^2K^2/\textbf{c}^2\}$, so that we get \eqref{p=sump} from \eqref{e=sume},
\eqref{estsumheatfinal} and
\eqref{e-p=inte} as in \cite[(3.20),\,(3.21)]{MM15}.
\end{proof}

\begin{remark}
Since the general version of Theorem \ref{thm:3.2new23}
established in \cite[Theorem\,1]{MM15} holds in the general symplectic
case, the proof of Theorem \ref{th-cov} extends verbatim to the
general symplectic case, thus extending \cite[Theorem\,2]{MM15} to
the case of manifolds with finite volume.
\end{remark}


\section{Bohr-Sommerfeld submanifolds}\label{BSsec}

Let $(X,J,g^{TX})$ be a complete Hermitian orbifold of complex
dimension $n$, together with a positive
holomorphic Hermitian proper 
orbifold line bundle $(L,h^L)$ satisfying \eqref{RL>eps}
and a holomorphic Hermitian proper 
orbifold vector bundle $(E,h^E)$. We set
\begin{equation}\label{preqBS}
    \om:=\frac{\sqrt{-1}}{2\pi}R^L\,.
\end{equation}
In this section, we recall the setting and results of \cite{Ioo18b}, which
are going to be the basic tools for the sequel.

Recall that an immersed submanifold $\iota:\Lambda\rightarrow X$
is said to be \emph{isotropic} if $\iota^*\om=0$. Let
$|\cdot|_{L}$ and $\nabla^{L}$ be the norm and connection
induced on the pullback $\iota^*L$ over $\Lambda$
by the Hermitian metric and Chern connection of $(L,h^L)$.
Note that by \eqref{preqBS}, the condition
$\iota^*\om=0$ implies that $\nabla^{L}$ is \emph{flat} over
$\Lambda$. This observation motivates the following definition.

\begin{definition}\label{BS}
A properly immersed oriented isotropic submanifold
$\iota:\Lambda\rightarrow X$ is said to satisfy the
\emph{Bohr-Sommerfeld condition} for $(L,h^L)$ if there exists a
non-vanishing smooth section
$\zeta\in\cC^{\infty}(\Lambda,\iota^* L)$
satisfying
\begin{equation}\label{nabs=0}
\nabla^{L} \zeta=0\;.
\end{equation}
Taking $\zeta$ such that $|\zeta(x)|_{L}=1$ for any $x\in\Lambda$,
the data of $(\Lambda,\iota,\zeta)$ is called a
\emph{Bohr-Sommerfeld submanifold} of $(X,L)$.
\end{definition}

As discussed in \cite[\S\,5.2]{Ioo18b}, in the case when
$X$ is an orbifold,
we consider \emph{singular immersions} $\iota:\Lambda\rightarrow X$
in the sense of
\cite[Def.\,5.6]{Ioo18b}, meaning that $\iota$ is a continuous map
from a smooth manifold $\Lambda$ into $X$, immersive on the smooth
locus of $X$ and such that the preimage of $\iota(\Lambda)$ on
orbifold charts is a union of cleanly intersecting submanifolds,
so that each pairwise intersection is also a submanifold whose tangent
space coincides with the intersection of the tangent spaces of the intersecting submanifolds.
Then the notion of a Bohr-Sommerfeld manifold extends tautologically
in this context.

Let $K_X:=\det(T^{*(1,0)}X)$ be the canonical line
bundle of $X$, and let $h^{K_X}$ be equal to
the Hermitian metric induced by $g^{TX}$. Denote by
$R^{K_{X}}$ the curvature of the Chern connection on 
$(K_{X}, h^{K_{X}})$. The following basic result can
found for instance in \cite[Lemma 6.1]{Ioo18b}.

\begin{lemma}
\label{BSgeod}
Assume $(K_X,h^{K_X})$ is positive, $\dim_\C X=1$ and that
$\frac{\sqrt{-1}}{2\pi} R^{K_X}=
c\,g^{TX}(J\cdot,\cdot)$
for some $c>0$. Let $\gamma:S^1\rightarrow X$ be a closed
geodesic loop of $X$.
Then $\gamma$ satisfies the Bohr-Sormmerfeld
condition for $(K_X,h^{K_X})$.

Furthermore, if $(L,h^L)$ is a holomorphic Hermitian line bundle
satisfying $(L^m,h^{L^m})\simeq (K_X,h^{K_X})$ for some
$m\in\N^*$,
then $\gamma^m:=\gamma\circ m:S^1\rightarrow X$
satisfies the Bohr-Sormmerfeld condition for $(L,h^L)$, where
$m:S^1\rightarrow S^1$ is the standard $m$-cover of $S^1$.
%
\end{lemma}
\begin{proof}
First recall that on a Riemannian orbifold, the finite isotropy group
of an orbifold chart acts by isometries, so that
a geodesic loop $\gamma:S^1\rightarrow X$ indeed defines
a singular immersion in the sense of \cite[Def.\,5.6]{Ioo18b}, i.\,e.,
the preimage of $\gamma$ in orbifold charts is a union of cleanly
intersecting subamnifolds.

Set now $\om=\frac{\sqrt{-1}}{2\pi}R^{K_X}$.
Recall that for $\om=c\,g^{TX}(J\cdot,\cdot)$ and for
$h^{K_X}$ induced by $g^{TX}$ up to a multiplicative constant, the
Chern connection $\nabla^{K_X}$ of $(K_X,h^{K_X})$ is induced by
the Levi-Civita connection $\nabla^{TX}$ on $TX$.
For any $\theta\in S^1$, let $\dot\gamma_\theta\in
T_{\gamma(\theta)}X$ be
the vector tangent to the curve $\gamma:S^1\rightarrow X$, and
write $\zeta\in\cC^{\infty}(S^1,\gamma^*K_X)$ for the induced
section of $K_X$ over $S^1$.
As $\gamma$ is geodesic, we know that
$\nabla^{TX}_{\dot\gamma}\dot\gamma=0$, so that
$\nabla^{K_X}_{\dot\gamma}\zeta=0$. Thus $\zeta$ satisfies
\eqref{nabs=0} and $\gamma:S^1\rightarrow X$
satisfies the Bohr-Sommerfeld condition for $(K_X,h^{K_X})$.


Let now $(L,h^L)$ be such that
$(L^m,h^{L^m})\simeq (K_X,h^{K_X})$ for some $m\in\N^*$.
The fact that $\gamma:S^1\rightarrow X$ satisfies
the Bohr-Sommerfeld
condition for $(K_X,h^{K_X})$ is equivalent to the fact that
the connection induced by $\nabla^{K_X}$ on $\gamma^*K_X$
is trivial along $S^1$. This implies in particular that
the holonomy of the Chern connection $\nabla^L$ of $(L,h^L)$ is trivial when pulled back on
the standard $m$-cover of $S^1$. This proves the second assertion.
\end{proof}

If $\zeta$ is any section of $\iota^*L$,
then we write $\zeta^p$ for the $p$-th power of $\zeta$ defined as
a section of $\iota^*L^p$. If additionally $f$ is a section of
$\iota^*E$, then we write $\zeta^pf$ for the induced tensor product in
$\iota^*(L^p\otimes E)$.
Let $(\Lambda,\iota,\zeta)$ be a compact
Bohr-Sommerfeld submanifold of $(X,L)$,
and consider a section $f\in\cC^{\infty}(\Lambda,\iota^*E)$.

\begin{definition}\label{Lagstate}
The \emph{isotropic state} associated with $(\Lambda,\iota,\zeta)$ and $f$
is the family of sections
$\{s_{\Lambda,p}\in H^0_{(2)}(X,L^p\otimes E)\}_{p\in\N^*}$ 
defined for any $x\in X$ by the formula
\begin{equation}\label{5.3}
s_{\Lambda,p}(x)=\int_{\Lambda} P_p(x,\iota(y))
.\zeta^pf(y)dv_{\Lambda}(y)\;,
\end{equation}
where $dv_\Lambda$ is the Riemannian volume form of
$(\Lambda,\iota^*g^{TX})$.
\end{definition}

The following proposition gives a characterization of isotropic
states in term of their reproducing properties.

\begin{proposition}\label{proprepgal}
For any $p\in\N^*$, the section
$s_{\Lambda,p}\in H^0_{(2)}(X,L^p\otimes E)$
of Definition \ref{Lagstate} is the unique element of
$H^0_{(2)}(X,L^p\otimes E)$
such that for any $s\in H^0_{(2)}(X,L^p\otimes E)$, we have
\begin{equation}\label{rep}
\big\langle s,s_{\Lambda,p}\big\rangle=
\int_{\Lambda} \big\langle s(\iota(x)),\zeta^pf(x)
\big\rangle_{L^p\otimes E}\ dv_{\Lambda}(x)\;.
\end{equation}
\end{proposition}
\begin{proof}
Recall that the orthogonal projection $P_p$ on 
$H^0_{(2)}(X,L^p\otimes E)$
is self-adjoint with respect to the $L^2$-Hermitian product,
and restricts
to the identity of $H^0_{(2)}(X,L^p\otimes E)$ for any $p\in\N^*$.
Then
for any $s\in H^0_{(2)}(X,L^p\otimes E)$ and as $\Lambda$ is compact,
we get from (\ref{lm2.0}), \eqref{5.3} and Fubini's theorem,
\begin{equation}\label{comprep}
\begin{split}
\big\langle s,s_{\Lambda,p}\big\rangle 
& =\int_X\left<s(y),\int_{\Lambda} P_p(y,\iota(x)).
\zeta^pf(x)dv_{\Lambda}(x)\right>_{L^p\otimes E} dv_X(y)\\
& =\int_{\Lambda} \left<\int_X P_p(\iota(x),y)s(y)dv_X(y),
\zeta^pf(x)\right>_{L^p\otimes E} dv_{\Lambda}(x)\\
& =\int_{\Lambda} \big\langle s(\iota(x)),\zeta^pf(x)
\big\rangle_{L^p\otimes E}\ dv_{\Lambda}(x)\;.
\end{split}
\end{equation}
The uniqueness comes from the fact that an element of a Hilbert
space is uniquely characterized by the values of its Hermitian product
with all other elements.
\end{proof}

Set $g^{TX}_\om=\om(\cdot,J\cdot)$ and 
$g^{T\Lambda}_\om=\iota^{\ast}g^{TX}_\om$. Let
$dv_{X,\om}$ and $dv_{\Lambda,\om}$ be the Riemannian volume forms
of $(X,g^{TX}_\om)$ and $(\Lambda,g^{T\Lambda}_\om)$, respectively.
Denote by $\|\cdot\|$ the norm on $L^{2}(X, L^{p}\otimes E)$
induced by (\ref{lm2.0}) and
by $|\cdot|_{\iota^{\ast}E}$  the norm on $\iota^{\ast}E$
over $\Lambda$ induced by $h^{E}$.
The following estimate on the norm of isotropic
states is the result of \cite[Theorem 5.3]{Ioo18b}.

\begin{theorem}\label{theonorme}
There exists $b_r\in\R$ for all $r\in\N$, such that for any
$k\in\N$ and as $p\rightarrow +\infty$,
\begin{equation}\label{norme}
\norm{s_{\Lambda,p}}^2=p^{n-\frac{\dim\Lambda}{2}}
\sum_{r=0}^k p^{-r} b_r+O(p^{n-\frac{\dim\Lambda}{2}-(k+1)})\;,
\end{equation}
with the first coefficient given by the formula
\begin{equation}\label{b0norme}
b_0=2^{\frac{\dim\Lambda}{2}}\int_{\Lambda}|f|_{\iota^*E}^2
\frac{dv_{X,\om}}{dv_X}\frac{dv_{\Lambda}}{dv_{\Lambda,\om}}
dv_{\Lambda}.
\end{equation}
In particular, if $f$ does not vanish identically,
there exists $p_0\in\N^*$ such that for all $p\geq p_0$, the section
$s_{\Lambda,p}\in H^0_{(2)}(X,L^p\otimes E)$
does not vanish identically.
\end{theorem}

In the case when $X$ is an orbifold,
the formula \eqref{b0norme} holds
only if the image of $\Lambda$ is not contained in the singular set
of $X$. Otherwise, we need to multiply the right hand side of
\eqref{b0norme} by the multiplicity of $\Lambda$ in $X$ as defined
in \cite[Def.\,5.6]{Ioo18b}.

In the case $X$ smooth and compact with $c_1(TX)$ even,
for $g_\om^{TX}=\om$ and for $\Lambda$ Lagrangian,
Theorem \ref{theonorme}
is the result of \cite[Theorem 3.2]{BPU95}.
Furthermore, the
expansion analogous to \eqref{norme} is only with half-integer powers
of $p$ in \cite[(85)]{BPU95} instead of integer powers as in
the expansion \eqref{norme}.
%

\section{Poincar\'e series on covering manifolds}
\label{s5.2}

Let $(\wi{X},\wi{J},g^{T\wi{X}})$ be a complete
Hermitian manifold of complex
dimension $n$, together with a positive
holomorphic Hermitian line bundle $(\wi{L},h^{\wi{L}})$
satisfying \eqref{RL>eps} and a holomorphic Hermitian
vector bundle  $(\wi{E},h^{\wi{E}})$, and
assume that these data have bounded geometry
in the sense
of Definition \ref{bndedgeomdef}.
Let $\Gamma$ be a discrete subgroup
of the group of holomorphic isometries of
$(\wi{X},\wi{J},g^{T\wi{X}})$ which lifts to
holomorphic Hermitian actions on $(\wi{L},h^{\wi{L}})$ and
$(\wi{E},h^{\wi{E}})$, and assume that the volume of
$X:=\wi{X}/\Gamma$ is finite.

Consider now a subgroup $\Gamma_0\subset\Gamma$. The goal
of the method of relative Poincar\'e series is to produce
$\Gamma$-invariant holomorphic sections of $\wi{L}^p\otimes \wi{E}$
from $\Gamma_0$-invariant holomorphic
sections, by averaging over the set
$\Gamma_0\backslash\Gamma$ of cosets
$[g]:=\{g_0g\in\Gamma:g_0\in\Gamma_0\}$, for all $g\in\Gamma$.
When $\Gamma_0=\{1\}$, we obtain the classical Poincar\'e series,
which are a basic tool in the theory of automorphic forms.
As emphasized for example in \cite[II.\,Chap.\,7]{Kol95},
the Bergman kernel can be
used in this context to construct a family of Poincar\'e series
parametrized by the points of $X=\wi{X}/\Gamma$, and we present
here a generalization which replaces the role of a point by an
isotropic submanifold of $X$.

Consider the quotient $X_0:=\wi{X}/\Gamma_0$
as a complex Riemannian orbifold $( X_0, J_0,g^{T X_0})$,
together with the holomorphic
Hermitian proper bundles $(L_0,h^{L_0})$ and $(E_0,h^{E_0})$
induced by
$(\wi{L},h^{\wi{L}})$ and $(\wi{E},h^{\wi{E}})$ in the quotient,
and let $(\Lambda,\iota_0,\zeta)$ be a compact Bohr-Sommerfeld
submanifold of $(X_0,L_0)$. We write
$\wi{\iota}:\wi{\Lambda}\rightarrow\wi{X}$ for the pullback of
$\iota_0:\Lambda\rightarrow X_{0}$ by the canonical projection
$\pi_0:\wi{X}\rightarrow X_0$. Then the action of $\Gamma_0$
on $\wi{X}$ induces one on $\wi{\Lambda}$ such that
$\Lambda=\wi{\Lambda}/\Gamma_0$, and
$\wi{\iota}:\wi{\Lambda}\rightarrow\wi{X}$ satisfies the
Bohr-Sommerfeld condition for $(\wi{L},h^{\wi{L}})$,
with unitary flat section
$\wi{\zeta}\in\cC^{\infty}(\wi\Lambda,\wi\iota^*\wi{L})$
inducing $\zeta$ in the quotient.
Note that $\wi{\Lambda}$ is not compact in general.

On the other hand, let $\pi:\wi{X}\rightarrow X:=\wi{X}/\Gamma$
be the canonical projection, and let $(L,h^L)$ and $(E,h^E)$
be the holomorphic Hermitian proper vector bundles induced by
$(\wi{L},h^{\wi{L}})$ and $(\wi{E},h^{\wi{E}})$ in the quotient.
There is then an induced covering $\hat\pi:X_0\rightarrow X$ such
that $\hat\pi\circ\pi_0=\pi$, and the proper immersion
$\iota:=\hat\pi\circ\iota_0:\Lambda\rightarrow X$ satisfies
the Bohr-Sommerfeld condition for $(L,h^L)$, with the same
flat unitary section $\zeta\in\cC^\infty(\Lambda,\iota^*L)$
via the induced natural isomorphism $\iota_0^*L_0\simeq\iota^*L$.

%
%
\begin{theorem}\label{t5.6}
For any
$f\in\cC^\infty(\Lambda,\iota_0^*E_0)$, there exists
$p_0\in\N^*$ such that for any $p\geq p_0$, the formula
\begin{equation}\label{isostatetil}
\wi{s}_{\wi{\Lambda},p}(x)=\int_{\wi{\Lambda}}\wi{P}_p(x,
\wi{\iota}(y)).\wi{\zeta}^p\wi{f}
(y)dv_{\wi{\Lambda}}(y)
\end{equation}
defines a $\Gamma_0$-invariant holomorphic section of
$\wi{L}^p\otimes \wi{E}$ for any $x\in\wi{X}$,
and the formula
\begin{equation}\label{rPsgalfla}
\wi{s}_{\Lambda,p}(x)=\sum_{[g]\in\Gamma_0\backslash\Gamma}
g^{-1}.\wi{s}_{\wi{\Lambda},p}(g.x),
\end{equation}
defines a $\Gamma$-invariant holomorphic section of
$\wi{L}^p\otimes \wi{E}$,
where the convergence is absolute and uniform for $x$ in any
compact subset of $\wi{X}$.
The section on $X$ induced in the quotient
by \eqref{rPsgalfla} belongs to
$H^0_{(2)}(X,L^p\otimes E)$, and is characterized in
$H^0_{(2)}(X,L^p\otimes E)$ by the reproducing property
\eqref{rep}.

Moreover, there exists $p_1\geqslant p_0$ such that 
the sections \eqref{isostatetil} and \eqref{rPsgalfla} do
not vanish identically for all $p\geqslant p_1$.
\end{theorem}

\begin{proof}
Recall that $\wi{P}_p(\cdot,\cdot)$ denotes the Bergman kernel of
$\wi{L}^p\otimes \wi{E}$ over $\wi{X}$ with respect to
$dv_{\wi{X}}$. Let $P_{0,p}(\cdot,\cdot)$ be the Bergman
kernels of $L_0^p\otimes E_0^p$
over $X_0$ with respect to the
Riemannian volume form of $(X_{0}, g^{TX_{0}})$.
Let $\{s_{0,p}\}_{p\in\N^*},\,\{s_{\Lambda, p}\}_{p\in\N^*}$ be the
isotropic states over $X_0,\,X$ associated with
$(\Lambda,\iota_0,\zeta),\,(\Lambda,\iota,\zeta)$
and $f\in\cC^\infty(\Lambda,\iota_0^*E_{0})$ via the natural
isomorphism $\iota_0^*E_0\simeq\iota^*E$, i.\,e., for $x\in X_{0}$,
\begin{equation}\label{}
s_{0,p}(x)=\int_{\Lambda} P_{0,p}(x,\iota_{0}(y))
.\zeta^pf(y)dv_{\Lambda}(y),
\end{equation}
and for $x\in X$, $s_{\Lambda, p}(x)$ is given by (\ref{5.3}).

Recall that $\Gamma$ acts on $\wi{X}$ by holomorphic
isometry, so that
for any $g\in\Gamma,\,x,\,y\in\wi{X}$ and $\zeta\in\wi{L}_y$, we have
\begin{equation}\label{Pinv}
(g^{-1},1).\wi{P}_p(g.x,g.y)g.\zeta=\wi{P}_p(x,y)\zeta\;.
\end{equation}
By Theorem \ref{th-cov} and (\ref{Pinv}),
there exists $p_0\in\N^*$ such that the following
formulas hold for any $x,\,y\in\wi{X}$ and $p\geqslant p_0$,
\begin{equation}
\begin{split}\label{Pmoy}
\pi_0^*P_{0,p}(\pi_0(x),\pi_0(y))=\sum_{g_0\in\Gamma_0}
(g_0^{-1},1).\wi{P}_p(g_0.x,y)\;,
\\
(\pi^{\ast}P_{p})(x, y)=\pi^*P_p(\pi(x),\pi(y))
=\sum_{g\in\Gamma} (g^{-1},1).\wi{P}_p(g.x,y)\;,
\end{split}
\end{equation}
with uniform convergence in $x, y\in K\times\wi{X}\cup\wi{X}\times K$ for any compact
$K\subset\wi{X}$. In particular,
for any $x\in X_0$ and $p\in\N^*$, we have
\begin{equation}
\begin{split}
\pi_0^*s_{0,p}(x)
&=\int_{\Lambda} \sum_{g_0\in\Gamma_0}
(g_0^{-1},1).\wi{P}_p(g_0.x,\iota_0(y))).
\zeta^p f(y)\,dv_{\Lambda}(y)\\
&=\sum_{g_0\in\Gamma_0}\int_{\Lambda} \wi{P}_p(x,g_0^{-1}.
\iota_0(y))
g_0^{-1}.\zeta^p f(y)\,dv_{\Lambda}(y)\\
&=\int_{\wi\Lambda} \wi{P}_p(x,\wi{\iota}(\wi{y})).
\wi{\zeta}^p\wi{f}(\wi{y})\,dv_{\wi{\Lambda}}(\wi{y})\;,
\end{split}
\end{equation}
where we identified $x\in X_0$ with a lift to $\wi{X}$.
Thus we have $\pi_0^*s_{0,p}=\wi{s}_{\wi{\Lambda},p}$, and
the integral \eqref{isostatetil} converges uniformly
in any compact subset of $\wi{X}$.
This shows the first assertion.

On the other hand, using again \eqref{Pinv} and \eqref{Pmoy},
for any $x\in\wi{X}$ and $p\in\N^*$, we have
\begin{equation}
\begin{split}
\pi^*s_{\Lambda, p}(x)
&=\int_{\Lambda}\sum_{[g]\in\Gamma_0\backslash\Gamma}
\sum_{g_0\in\Gamma_0}(g^{-1}g_0^{-1},1).\wi{P}_p(g_0g.x,y)
.\zeta^pf(y)dv_{\Lambda}(y)\\
&=\sum_{[g]\in\Gamma_0\backslash\Gamma}\int_{\Lambda}
\sum_{g_0\in\Gamma_0}(g^{-1},1).\wi{P}_p(g.x,g_0^{-1}.y)
g_0^{-1}.\zeta^pf(\wi{y})dv_{\Lambda}(\wi{y})\\
&=\sum_{[g]\in\Gamma_0\backslash\Gamma}\int_{\wi\Lambda}
(g^{-1},1).\wi{P}_p(g.x,\wi{y}).
\wi{\zeta}^p\wi{f}(\wi{y})dv_{\wi\Lambda}(\wi{y})\;,
\end{split}
\end{equation}
where we identified $x\in X_0$ with a lift to $\wi{X}$.
Thus we have $\pi^*s_{\Lambda, p}=\wi{s}_{\Lambda, p}$, 
and the series in \eqref{rPsgalfla}
converges absolutely and uniformly on compact subsets.
In particular, we know from Proposition \ref{proprepgal} that
$s_{\Lambda, p}$ is
characterized by the reproducing property \eqref{rep}, and
Theorem \ref{theonorme} shows that \eqref{rPsgalfla} does not vanish
for $p\in\N^*$ large enough.
\end{proof}


%


Let now $K_{\wi{X}}$, $K_X$ be the canonical line bundles of
$\wi{X}$, $X=\wi{X}/\Gamma$ endowed with the Hermitian
metrics $h^{K_{\wi{X}}}$, $h^{K_X}$ induced by
$g^{T\wi{X}}$, $g^{TX}$ up to a common multiplicative constant.
Then $K_{\wi{X}}=\pi^* K_X$, and as $g^{T\wi{X}}=\pi^*g^{TX}$,
we get $h^{K_{\wi{X}}}=\pi^*h^{K_X}$.
Then if $\gamma:S^1\rightarrow X$ is a geodesic loop for $g^{TX}$,
there exists
$g_0\in\Gamma$ and a geodesic $\wi\gamma:\wi{S^1}\rightarrow\wi{X}$
covering $\gamma$
such that $\wi\gamma(t+1)=g_0.\wi\gamma(t)$ for all $t\in\R$,
where $\wi{S^1}=\R$ if $g_0$ generates a free subgroup
$\Gamma_0\subset\Gamma$, and $\wi{S^1}=S^1$ otherwise. 
In the first case, we say that $g_0\in\Gamma$ is \emph{loxodromic}
and that $\wi{\gamma}$ is \emph{generated} by $g_0$.
Then if the geodesic $\gamma:S^1\rightarrow X$ satisfies the Bohr-Sommerfeld condition,
we are exactly in the situation of Theorem \ref{t5.6},
where $\Gamma_0$ is the subgroup of $\Gamma$
generated by $g_0\in\Gamma$ and
$(L,h^L)$ is as in Lemma \ref{BSgeod}.
In the case
$\dim_\C X=1$ considered in Lemma \ref{BSgeod}, the reproducing formula \eqref{rep}
for the associated isotropic state is called in
\cite[\S\,3.2,\,(2)]{FK01} the \emph{period
formula} along the geodesic $\gamma$.

\section{Poincar\'e series on Hermitian symmetric spaces}
\label{s5.3}
Theorem \ref{t5.6} is especially interesting in the case
$\wi{X}$ admits a large
automorphism group and its Bergman kernel
has an explicit form, so that this method produces
a wide variety of relative Poincar\'e series
which are explicitly computable. This is typically the case for
\emph{Hermitian symmetric spaces}, which are the spaces
considered in the theory of holomorphic automorphic forms.

Let $\wi{X}$ be a complex manifold of dimension $n$, and let
$G$ be a Lie group acting transitively on $\wi{X}$ by
biholomorphisms. Assume that there exists a nowhere
vanishing holomorphic $n$-form $\nu\in H^0(\wi{X},K_{\wi{X}})$, and
set
\begin{equation}\label{cocycle}
\begin{split}
g.\nu_z&=(dg^{-1})^*.\,\nu_z=:\lambda(g,z)\nu_{g.z}
\end{split}
\end{equation}
for the induced action of $g\in G$ on $K_{\wi{X}}$,
where $\lambda(g,z)\in\C$ for any $z\in\wi{X}$.
Define a volume form $d\mu$ on $\wi{X}$ by the formula
\begin{equation}\label{mu}
d\mu:=\left(\frac{\sqrt{-1}}{2}\right)^n(-1)^{\frac{n(n-1)}{2}}
\nu\wedge\overline{\nu}\;.
\end{equation}
Let $L^2(\wi{X},\C,d\mu)$ be the $L^2$-space associated with $d\mu$
and let $H^0_{(2)}(\wi{X},\C,d\mu)$ be the
subspace of holomorphic $L^2$-functions.
Let $P_{d\mu}(\cdot,\cdot)\in\cC^{\infty}(\wi{X}\times\wi{X},\C)$
be the Schwartz kernel of the orthogonal
projection $P_{d\mu}:L^2(\wi{X},\C,d\mu)\rightarrow
H^0_{(2)}(\wi{X},\C,d\mu)$ with respect to $d\mu$.
It is uniquely characterized by the following conditions:
\begin{itemize}
\item $P_{d\mu}(z,w)$ is holomorphic in $z\in\wi{X}$ and
antiholomorphic in $w\in\wi{X}$,
\item $P_{d\mu}(z,w)=\overline{P_{d\mu}(w,z)}$ for all
$z,\,w\in\wi{X}$,
\item For any $f\in H^0_{(2)}(\wi{X},\C,d\mu)$ and $z\in\wi{X}$,
the following formula holds,
\begin{equation}
f(z)=\int_{\wi{X}}P_{d\mu}(z,w)f(w)d\mu(w)\;.
\end{equation}
\end{itemize}
We make the following
assumptions on $(\wi{X},d\mu)$:
\begin{equation}\label{cond}
\begin{split}
&\text{For any $x\in\wi{X}$, there exists $f\in H^0_{(2)}(\wi{X},\C,d\mu)$
with $f(x)\neq 0$, and}\\
&\text{for any $v\in T_x\wi{X}$, there exists
$f\in H^0_{(2)}(\wi{X},\C,d\mu)$ with $f(x)= 0$
and $df.v\neq 0$}\;.
\end{split}
\end{equation}
Note that if $\wi{X}$ is a \emph{bounded symmetric
domain} of $\C^n$
and taking $\nu=dz$,
then $d\mu$ is the Lebesgue measure and these assumptions hold
for $(\wi{X},d\mu)$. In general, domains of $\C^n$ together with
$\nu=dz$ always
satisfy the assumptions \eqref{cond}, and such domains
with a transitive biholomorphic action of a Lie group $G$
are called \emph{Hermitian symmetric domains}.

As explained in \cite[Theorem 5.1]{Kob59},
under the assumptions \eqref{cond},
we have $P_{d\mu}(z,z)>0$ for all $z\in\wi{X}$,
and the formula
\begin{equation}
\om_z:=\sqrt{-1}\partial\overline{\partial}\log P_{d\mu}(z,z)\quad
\text{for all}~~z\in\wi{X}\;,
\end{equation}
defines a $G$-invariant K\"ahler form on $\wi{X}$.
The induced metric
$g^{T\wi{X}}$ is called the \emph{Bergman metric}, and
we write $dv_{\wi{X}}$ for the Riemannian volume form of
$(\wi{X},g^{T\wi{X}})$. Let $|\cdot|_{K_{\wi{X}}}$ be the Hermitian
norm on the canonical line bundle $K_{\wi{X}}$ defined by
the formula
\begin{equation}\label{nunu=om}
d\mu=:|\nu|_{K_{\wi{X}}}^2dv_{\wi{X}}\;.
\end{equation}
Then by \eqref{mu}, the associated Hermitian metric
$h^{K_{\wi{X}}}$ on $K_{\wi{X}}$
is $G$-invariant and equal to the metric induced by $g^{T\wi{X}}$ 
divided by $2^n$. Furthermore, as explained in the proof of
\cite[IV.\,\S\,1,\,Proposition 3]{Mok89},
we have
\begin{equation}\label{KE}
\sqrt{-1}R^{K_{\wi{X}}}=\om\;,
\end{equation}
where $R^{K_{\wi{X}}}$ is the curvature of the Chern connection of
$(K_{\wi{X}},h^{K_{\wi{X}}})$.

Consider now the $L^2$-Hermitian product on
$\cC^{\infty}(\wi{X},K_{\wi{X}})$ induced by
$h^{K_{\wi{X}}}$ and $g^{T{\wi X}}$ as in
\eqref{lm2.0}, and let $L^2(\wi{X},K_{\wi{X}})$ be the associated
Hilbert space. By \eqref{nunu=om}, there is an isometry
\begin{equation}
\begin{split}
L^2(\wi{X},\C,d\mu)&\longrightarrow L^2(\wi{X},K_{\wi{X}})\\
f&\longmapsto f\nu\;,
\end{split}
\end{equation}
identifying $H^0_{(2)}(\wi{X},\C,d\mu)$ with
$H^0_{(2)}(\wi{X},K_{\wi{X}})$.
Denote by $P^{K_{\wi{X}}}(\cdot,\cdot)$
is the Bergman kernel of $H^0_{(2)}(\wi{X},K_{\wi{X}})$ with respect
to $dv_{\wi{X}}$, then
\begin{equation}\label{P=Pmu}
P^{K_{\wi{X}}}(z,w)=P_{d\mu}(z,w)\nu_z\overline{\nu}_w\;,
\end{equation}
via the identification $K_{\wi{X}}^*\simeq\overline{K_{\wi{X}}}$
induced by $h^{K_{\wi{X}}}$.

Finally, let $(\wi{L},h^{\wi{L}})$ be a holomorphic Hermitian
line bundle over $\wi{X}$ together with a lift of the action
of $G$ on $\wi{X}$, and assume that there is an identification
$(\wi{L}^q,h^{\wi{L}^q})\simeq (K_{\wi{X}}^p,h^{K_{\wi{X}}^p})$ 
compatible with the action of $G$, for some $p,\,q\in\N^*$.
Let $P^{\wi{L}}(\cdot,\cdot)$ be the Bergman kernel of
$H^0_{(2)}(\wi{X},\wi{L})$ with respect to $dv_{\wi{X}}$. Then
as explained in \cite[Ex.\,7.7]{Kol95},
there is a constant $c_{p,q}>0$ such that
\begin{equation}\label{explicitfla}
P^{\wi{L}^q}(z,w)=c_{p,q}P_{d\mu}(z,w)^p\nu_z^p~\overline\nu_w^p,
\end{equation}
for all $z,\,w\in\wi{X}$.

As $G$ acts transitively on $\wi{X}$ and by
\eqref{KE}, we see that
$(\wi{X},\wi{J},g^{T\wi{X}})$ and $(\wi{L},h^{\wi{L}})$ as above
satisfy the hypotheses of Theorem \ref{th-cov}.
Thus we can apply Theorem \ref{t5.6} to this situation, for any
discrete subgroup $\Gamma\subset G$ acting effectively and
satisfying $\text{Vol}(\wi{X}/\Gamma)<+\infty$.
In particular, from Theorem \ref{th-cov},
\eqref{cocycle} and \eqref{explicitfla},
we get the
existence of $p_0\in\N^*$ such that for all $p\geq p_0$, the Poincar\'e series
\begin{equation}\label{Poincfla}
\pi^*P^{K_X^p}(\pi(z),\pi(w))=c_p\,\left(\sum_{g\in\Gamma}
\lambda(g,z)^{-p}P_{d\mu}(g.z,w)^p\right)\nu_z^p~\overline\nu_w^p,
\end{equation}
converges uniformly in $z, w\in K\times\wi{X}\cup\wi{X}\times K$ for any compact
$K\subset\wi{X}$, where $c_p>0$ for all $p\geq p_0$. In the case $\dim_\C X=1$, if a geodesic
$\wi{\gamma}:\R\rightarrow\wi{X}$ is generated by a
loxodromic element $g_0\in\Gamma$, we can use Theorem \ref{t5.6}
to construct the relative Poincar\'e series associated
to the induced geodesic loop $\gamma:S^1\rightarrow X$,
which satisfies the period formula \eqref{rep}, and these
can be explicitly computed in a number of cases which are presented
below. In all these cases, Theorem \ref{t5.6} shows
that there exists
$p_0\in\N^*$ such that these series
do not vanish identically for $p\geqslant p_0$ large enough.

\begin{remark}
In the following section, the action of $G$ on $(\wi{X},\wi{L})$
will not be effective, but instead will satisfy the property
that the natural projection
$\pi:G\rightarrow\text{Aut}(\wi{X})$ has finite kernel.
In that case, if $\Gamma\subset G$ is a discrete subgroup,
the quotient line bundle $L=\wi{L}/\Gamma$ is not a proper
orbifold line bundle in general, so that
Theorem \ref{th-cov} does not even make sense.
Instead, the properness of $L$ is equivalent to the following
condition on $\Gamma\subset G$,
\begin{equation}\label{condGam}
g\in\Gamma~~\text{acts trivially on}~~\wi{L}~~
\text{if and only if $g$ acts trivially on}~~\wi{X}\,.
\end{equation}
Under this assumption, the action of $\Gamma$ on
$(\wi{X},\wi{L})$
factors through the action of its image $\hat\Gamma$
via the projection
$\pi:G\rightarrow\text{Aut}(\wi{X})$.
Then we can work with $\hat\Gamma$ instead, and the associated
Poincar\'e series as in formula \eqref{Poincfla} only differs by a
multiplicative factor
equal to the cardinal of $\Ker[\pi:\Gamma\rightarrow\hat\Gamma]$.

Note that condition \eqref{condGam} is always satisfied for
$\wi{L}=K_{\wi{X}}$ or if $\Gamma$ acts effectively on $\wi{L}$.
In the following section, we will assume for simplicity that $\Gamma$
acts effectively, but all the results of Section \ref{s5.4}
actually hold under the more
general assumption \eqref{condGam}.
\end{remark}

\section{Main examples of Poincar\'e series}
\label{s5.4}

In this section, we present the main examples of Hermitian
symmetric domains considered in Section \ref{s5.3}, expliciting
formula \eqref{Poincfla} in concrete situations.
In particular, we exhibit some new examples
of relative Poincar\'e series over domains of $\C^n$
for which Theorem \ref{t5.6} shows that they
do not vanish identically.

\begin{example}\label{Hilbmodvar}
Consider the special linear group
\begin{equation}
\SL_2(\R)=\bigg\{g=
\begin{pmatrix}
a & b \\
c & d
\end{pmatrix}
~:~a,b,c,d\in\R,~ad-bc=1 \bigg\}
\end{equation}
acting on the Poincar\'e upper-half plane 
$\mathbb{H}=\left\{z=x+\sqrt{-1}y\in\C~:~y>0\right\}$ by the formula
\begin{equation}\label{graphie}
g.z:=\frac{az+b}{cz+d}\;.
\end{equation}
Fix $n\in\N^*$ and set $\wi{X}=\mathbb{H}^n$.
Then the group $\SL_2(\R)^n$ acts on $\mathbb{H}^n$ by
\eqref{graphie} on each summand. If $\Gamma\subset \SL_2(\R)^n$
is a discrete subgroup, the quotient $X:=\mathbb{H}^n/\Gamma$
is called a \emph{Hilbert modular variety}.
We are then in the situation discussed above, with
$\nu=dz:=dz_1\dots dz_n$,
the canonical section of $K_{\mathbb{H}^n}$
over $\mathbb{H}^n$. The action of $g\in \SL_{2}(\R)^{n}$
on $K_{\mathbb{H}^n}$ by pushforward as in \eqref{cocycle}
is given at $z\in\mathbb{H}^n$ by
\begin{equation}\label{actcan}
g.dz=\prod_{j=1}^n(c_jz_j+d_j)^{2}dz=:j(g,z)^{2}dz\;,
\end{equation}
where $a_j,\,b_j,\,c_j,\,d_j\in\R$ are the entries of the $j$-th
summand of $g=(g_1,\dots, g_n)\in \SL_2(\R)^n$ as a matrix of
$\SL_2(\R)$. The $\SL_2(\R)^n$-invariant metric
$h^{K_{\mathbb{H}^n}}$ on
$K_{\mathbb{H}^n}$ defined
as in \eqref{nunu=om} is given by the formula
\begin{equation}
|dz|_{K_{\mathbb{H}^n}}^2=\prod_{j=1}^n y_j^2\;,
\end{equation}
where $z=(x_1+\sqrt{-1}y_1,\dots, x_n+\sqrt{-1}y_n)\in\mathbb{H}^n$.
Let $\wi{L}$ be the trivial holomorphic bundle over $\mathbb{H}^n$,
with a nowhere
vanishing holomorphic section $\sigma\in H^0(\mathbb{H}^n,\wi{L})$.
We define a lift of the action of $\SL_2(\R)^n$ to $\wi{L}$
and a
$\SL_2(\R)^n$-invariant metric $h^{\wi{L}}$ by the formulas
\begin{equation}\label{actcanprod}
g.\sigma_z=j(g,z)\sigma_{g.z}~~\quad\text{and}~~\quad
|\sigma_z|_{\wi{L}}^2=\prod_{j=1}^n y_j\;.
\end{equation}
This gives a $\SL_2(\R)^n$-equivariant identification
$(\wi{L}^2,h^{\wi{L}^2})\simeq(K_{\mathbb{H}^n},
h^{K_{\mathbb{H}^n}})$
through which $\sigma^2=dz$. The Bergman kernel of
$H^0_{(2)}(\mathbb{H}^n,\wi{L}^p)$ for
$p\in\N^*,\,p>2$, can be computed from the proof of
\cite[II.\,\S\,1,\,Proposition 1.2]{Frei90} and is given by
\begin{equation}\label{Pdmu1}
P^{\wi{L}^p}(z,w)=\left(\frac{p-1}{4\pi}\right)^n
\prod_{j=1}^n\left(\frac{2\sqrt{-1}}
{z_j-\overline{w}_j}\right)^{p}\sigma_z^p\,\overline\sigma_w^p\;,
\end{equation}
for any $z,w\in\mathbb{H}^n$, via the identification of
$\overline{\sigma}$ with the metric dual of $\sigma$ for
$h^{\wi{L}}$.
Then we have the following result.
\begin{proposition}\label{Ex1}
Let $\Gamma\subset\SL_2(\R)^n$ be a discrete subgroup
acting effectively and satisfying
$\text{Vol}(\mathbb{H}^n/\Gamma)<+\infty$.
Then for any compact set $K\subset\mathbb{H}^n$,
there exists
$p_0\in\N^*$ such that for all $p\geq p_0$ and all $w\in K$,
the convergent series
\begin{equation}\label{thhyp}
\sum_{g\in\Gamma}j(g,z)^{-p}
\prod_{j=1}^n\left(\frac{2\sqrt{-1}}
{g_j.z_j-\overline{w}_j}\right)^{p}\;,
\end{equation}
does not vanish identically in
$z\in\mathbb{H}^n$.
\end{proposition}
\begin{proof}
The domain $\wi{X}=\mathbb{H}^n$ together with the nowhere
vanishing
holomorphic $n$-form $\nu=dz$ and the action of $G=\SL_2(\R)^{n}$
described by formula \eqref{actcan} is easily seen to satisfy
the assumptions \eqref{cond}.
As explained at the end
of Section \ref{s5.3}, this implies that
the domain $\mathbb{H}^n$ and the holomorphic
Hermitian line bundle $(\wi{L},h^{\wi{L}})$ described above
satisfy the hypotheses of Theorem \ref{th-cov}. Then taking
$X=\mathbb{H}^n/\Gamma$,
$L=\wi{L}/\Gamma$ and $\Lambda=\{w\}$, by
\eqref{Poincfla}, \eqref{actcan} and \eqref{Pdmu1},
the series \eqref{thhyp} identifies with the section
$\wi{s}_{\Lambda,p}\in H^0_{(2)}(\wi{X},\wi{L})$
of \eqref{rPsgalfla} up to a multiplicative constant. Thus
Theorem \ref{Ex1} is a consequence of
Theorem \ref{t5.6}.
\end{proof}

Under the hypotheses of Theorem \ref{Ex1}, the space
$H^0_{(2)}(X,L^p)$ identifies with the space
of \emph{Hilbert cusp forms of weight} $p$ via the identity
\begin{equation}
f(g.z)\,dz=f(z)\,g.dz=f(z)j(g,z)^2\,dz\,.
\end{equation}
for $\Gamma$-equivariant sections of $K_{\mathbb{H}^n}$ over
$\mathbb{H}^n$.
Then Theorem \ref{Ex1} improves the classical result,
which can be found
in \cite[I.\,Proposition 5.6]{Frei90}, that these series
are non-vanishing for an infinite number of $p\in\N^*$.

In the case $n=1$ and
$p$ even, so that we may replace $L$
by $K_X$, the relative Poincar\'e series associated with
geodesics as in Theorem \ref{t5.6} have been
considered by Katok in \cite{Kat85}. In particular,
she proves that the series associated with \emph{hyperbolic}
geodesics, i.\,e., geodesics generated by a loxodromic element
as defined at the end of Section \ref{s5.2},
generate the space of
cusp forms of weight $2p$ for all $p\in\N^*$.
The following result shows in turn that these series do not vanish
for $p\in\N^*$ large enough.

\begin{theorem}{\cite[Theorem 6.7,\,(6.18)]{Ioo18b}}
\label{Ex2}
Let $\Gamma\subset\SL_2(\R)$ be a discrete subgroup acting effectively
and satisfying
$\text{Vol}(\mathbb{H}/\Gamma)<+\infty$,
let $g_0\in\Gamma$ be a loxodromic element and let 
$\Gamma_0\subset\Gamma$ be the subgroup generated by $g_0$.
Then there exists $p_0\in\N^*$ such that for any
$p\geq p_0$,
the convergent series
\begin{equation}\label{RPShyp}
\wi{s}_{\gamma,p}(z)=\sum_{[g]\in\Gamma_0\backslash\Gamma}
j(g,z)^{-2p}\frac{1}{\left(c(g.z)^2+(d-a)(g.z)-b\right)^{p}}dz^p\;,
\end{equation}
does not vanish identically in $z\in\mathbb{H}$.
\end{theorem}
\begin{proof}
Consider the setting of Section \ref{s5.2}, with
$\wi{X}=\mathbb{H}$ and $\wi{L}=K_{\wi{X}}$, so that
$X=\mathbb{H}/\Gamma$ and $L=K_X$.
Let $\wi{\gamma}:\R\rightarrow\wi{X}$ be the geodesic
generated by $g_0\in\Gamma$, and write $\gamma:S^1\rightarrow X$
for the geodesic loop induced in the quotient,
considered as a compact Bohr-Sommerfeld
submanifold $(S^1,\gamma,\zeta)$ of $(X,K_X)$ as in
Lemma \ref{BSgeod}. Then by Theorem \ref{t5.6}, we get
$p_0\in\N$ such that
the associated section
$\wi{s}_{\gamma,p}\in H^0(\mathbb{H},K_{\mathbb{H}}^p)$
defined by \eqref{rPsgalfla}
does not vanish for $p\geq p_0$.

On the other hand,
by \cite[Proposition 4]{Kat85} and by
the explicit formula given in \cite[(1.3)]{Kat85},
the section of $H^0_{(2)}(X,K_X^p)$ induced by the formula
\eqref{RPShyp} for all $p\in\N^*$ satisfies the
characterizing property \eqref{rep} for $(S^1,\gamma,\zeta)$,
up to a multiplicative constant $C_p>0$.
By Proposition \ref{proprepgal}, we deduce that this section is
equal to $C_p\wi{s}_{\gamma,p}$, and thus does not vanish
identically.
\end{proof}
In the case of a smooth and compact
hyperbolic Riemann surface $(X,g^{TX})$ , this is the result of
\cite[Theorem 4.11]{BPU95}.
\end{example}

\begin{example}\label{hypex}
Let $n\in\N^*$, and consider the symplectic group
\begin{equation}
\Sp_{2n}(\R)=\bigg\{g=
\begin{pmatrix}
A & B \\
C & D
\end{pmatrix}
\in\GL_{2n}(\R)~:~A^T\,C,\,B^T\,D\text{ symmetric}\;,
~A^T\,D-C^T\,B=\Id_{\R^n} \bigg\}
\end{equation}
acting on the \emph{Siegel upper half space}
\begin{equation}
H_n=\{Z=X+\sqrt{-1}Y\in \M_{n}(\C)~:~Z\text{ symmetric},
~Y\text{ positive definite}\}
\end{equation}
by the formula
\begin{equation}
g.Z=(AZ+B)(CZ+D)^{-1}\;.
\end{equation}
We are then in the situation discussed above, with $\nu=dZ$
the canonical section of $K_{H_n}$ over $H_n$, seen as a domain
in $\C^{n(n+1)/2}$. The action of $g\in\Sp_{2n}(\R)$
on $K_{H_n}$ as in
\eqref{cocycle} is given at $Z\in H_n$ by
\begin{equation}\label{siegactcan}
g.dZ=\det(CZ+D)^{n+1}dZ=:J(g,Z)^{n+1}dZ\;,
\end{equation}
and the $\Sp_{2n}(\R)$-invariant metric
$h^{K_{H_n}}$ on $K_{H_n}$ defined as in \eqref{nunu=om} is given
by the formula
\begin{equation}
|dZ|_{K_{H_n}}^2=\det Y^{n+1}\;.
\end{equation}
Let $\wi{L}$ be the trivial holomorphic bundle over $H_n$,
with a nowhere
vanishing holomorphic section $\sigma\in H^0(H_n,\wi{L})$.
We define a lift of the action of $\Sp_{2n}(\R)$ to $\wi{L}$ and a
$\Sp_{2n}(\R)$-invariant metric $h^{\wi{L}}$ for all $Z\in H_n$ and
$g\in\Sp_{2n}(\R)$ by the formulas
\begin{equation}\label{siegactcanprod}
g.\sigma_Z=J(g,Z)\,\sigma_{g.Z}~~\quad\text{and}~~\quad
|\sigma_Z|_{L}^2=\det Y\;,
\end{equation}
This gives a $\Sp_{2n}(\R)$-equivariant identification
$(\wi{L}^{n+1},h^{\wi{L}^{n+1}})\simeq(K_{H_n},h^{K_{H_n}})$
through which $\sigma^{n+1}=dZ$.
The Bergman kernel of $H^0_{(2)}(\mathbb{H},\wi{L}^p)$
for $p\in\N^*,\,p>2n$, is computed for example in
\cite[III.\,Proposition 1]{Kli90} and is given by
\begin{equation}\label{Pdmusieg}
P^{\wi{L}^p}(Z,W)=
\frac{a_{n,p}}{\det(Z-\overline{W})^{p}}\sigma_Z^p\,
\overline\sigma_W^p\;,
\end{equation}
for any $Z,\,W\in H_n$, via the identification of
$\overline{\sigma}$ with the metric dual of $\sigma$ for
 $h^{\wi{L}}$. The constant $a_{n,p}>0$ can be found in
\cite[p.\,78]{Kli90}, up to a multiplicative
factor of $2^{n(n-1)/2}$ due to a different normalization for
the volume form in \cite[p.\,10]{Kli90}. Then we
have the following result.

\begin{proposition}\label{Ex3}
Let $\Gamma\subset\text{Sp}_{2n}(\R)$ be discrete subgroup acting
effectively and satisfying
$\text{Vol}(H_n/\Gamma)<+\infty$.
Then for any compact set $K\subset H_n$, there exists $p_0\in\N^*$ such that for any $p\geq p_0$
and any $W\in K$,
the convergent series
\begin{equation}
\sum_{g\in\Gamma}J(g,Z)^{-p}
\frac{1}{\det(g.Z-\overline{W})^{p}}\;,
\end{equation}
does not vanish identically in $Z\in H_n$.
\end{proposition}
\begin{proof}
By \eqref{Poincfla}, \eqref{siegactcan} and
\eqref{Pdmusieg}, this is a consequence of Theorem \ref{t5.6} with
$\wi{X}=H_n$, $\Lambda=\{W\}$ and $\wi{L}$ as above,
in the same way as in the proof of Theorem \ref{Ex1}.
\end{proof}
\end{example}

\begin{example}
Let $n\in\N^*$, and consider the Hermitian form $Q(\cdot,\cdot)$
on $\C^{n+1}$ given by
\begin{equation}
Q(z,w)=\sum_{k=1}^nz_k\overline{w}_k-z_0\overline{w}_0\;,
\end{equation}
for any $(z_0,\dots,z_n),\,(w_0,\dots,\,w_n)\in\C^{n+1}$.
Let $\SU(n,1)\subset\SL_{n+1}(\C)$ be the subgroup of
$\SL_{n+1}(\C)$ preserving $Q$.

Consider the standard affine chart on the
open set $\{z_0\neq 0\}\subset\C \mathbb{P}^n$, which identifies
$z=(z_1,\dots, z_n)\in\C^n$ to $[1:z_1:\dots:z_n]\in\C \mathbb{P}^n$
in homogeneous coordinates.
Then the induced action of $SU(n,1)$ on $\C \mathbb{P}^n$ preserves
the open unit ball $B_n\subset\C^n$ via these coordinates,
and is transitive on $B_n$.

Let $\wi{L}$ be the pullback of the tautological bundle
of $\C \mathbb{P}^n$ to $B_n$. Recall that the homogeneous polynomial
$P(z)=z_0$ can be seen as a section of its dual $\wi{L}^*$,
so that the dual $\sigma$ of $z_0$
is a nowhere vanishing holomorphic section of $\wi{L}$ over $B_n$.
The action of $SU(n,1)\subset\SL_{n+1}(\C)$ on $\C^{n+1}$
induces one on $\wi{L}$, which is easily seen to be given
for any $g\in\SU(n,1)$ and $z\in B_n$ by
\begin{equation}\label{ballactcan}
g.\sigma=(c.z+d)\sigma=:\mathfrak{J}(g,z)\sigma\;,
\end{equation}
where we wrote $$g=
\begin{pmatrix}
A & b \\
c^{T} & d
\end{pmatrix},$$
with $A\in M_n(\C),~b,\,c\in\C^n$ and $d\in\C$, and where
$c.z:=\sum_{j=1}^n c_jz_j$, with $z=(z_1,\dots,z_n)\in B_n$.

Let $\langle\cdot,\cdot\rangle$ be the standard Hermitian product
on $\C^n$ and let $|\cdot|$ be the associated norm. We define
a $\SU(n,1)$-invariant metric $h^{\wi{L}}$
on $\wi{L}$ at $z\in B_n$ using the formula
\begin{equation}
|\sigma_z|_L^2=(1-|z|^2).
\end{equation}
Then there is a canonical $\SU(n,1)$-equivariant
identification of holomorphic Hermitian
vector bundles
$(\wi{L}^{n+1},h^{\wi{L}^{n+1}})\simeq (K_{B_n},h^{K_{B_n}})$
under which $dz$ corresponds to $\sigma^{n+1}$, where $h^{K_{B_n}}$
is the metric on $K_{B_n}$ defined as in \eqref{nunu=om} with
$\nu=dz$.
As $B_n$ is a bounded symmetric domain, we are in fact
in the situation described
in Section \ref{s5.3}.

The Bergman kernel of $H^0_{(2)}(B_n,\wi{L}^p)$
is computed, for example, in
\cite[(8.10.2),\,(8.10.5)]{Kol95} and is given for any $p\in\N^*$
and $z,\,w\in B_n$ by
\begin{equation}\label{SUberg}
P^{\wi{L}^p}(z,w)=\frac{n!}{\pi^n}\begin{pmatrix}
n+p \\
n
\end{pmatrix}
\frac{1}{(1-\langle z,w\rangle)^p}\sigma_z^p\,
\overline{\sigma}_w^p\;,
\end{equation}
via the identification of
$\overline{\sigma}$ with the metric dual of $\sigma$ for
$h^{\wi{L}}$.
Then we have the following result.

\begin{proposition}\label{Ex4}
Let $\Gamma\subset\SU(n,1)$ be a discrete subgroup acting
effectively and satisfying
$\text{Vol}(B_n/\Gamma)<+\infty$. Then for any compact set $K\subset B_n$
there exists $p_0\in\N^*$ such that for any $p\geq p_0$ and $w\in K$,
the convergent series
\begin{equation}
\sum_{g\in\Gamma}\mathfrak{J}(g,z)^{-p}
\frac{1}{(1-\langle z,w\rangle)^p}\;,
\end{equation}
does not vanish identically in $z\in B_n$.
\end{proposition}
\begin{proof}
By \eqref{Poincfla}, \eqref{ballactcan} and \eqref{SUberg},
this is a consequence of Theorem \ref{t5.6} with
$\wi{X}=B_n$, $\Lambda=\{w\}$ and $\wi{L}$ as above,
in the same way as in the proof of Theorem \ref{Ex1}.
\end{proof}

As above, the space $H^0_{(2)}(X,L^p)$ identifies
with the space of \emph{cusp forms of weight} $p$.
When
$2(n+1)$ divides $p$,
so that we replace $L$ by $K_X^2$, the relative Poincar\'e series
associated with
geodesics as in Lemma \ref{BSgeod} have been considered by Foth and
Katok in \cite{FK01}. In particular, they prove that the
series associated with loxodromic geodesics
generate the space of cusp
forms of weight $2(n+1)k$ for all $k\in\N^*$.
If furthermore, all the eigenvalues of the loxodromic element $g_0\in\Gamma$ are real
and if the endpoints $x,\,y\in\partial B_n$ of the geodesic $\wi{\gamma}\subset  B_n$ generated by $g_0$
belong to $\R^n\subset\C^n$, 
then Barron showed in \cite[\S\,2.2]{Bar18} that the geodesic $\wi{\gamma}\subset  B_n$
satisfies the Bohr-Sommerfeld condition in the sense of Definition \ref{BS}.
The following result then shows once again that these series do not
vanish for $p\in\N^*$ large enough.
In the case of a smooth, compact manifold $(X,g^{TX})$,
the result is due to Barron \cite[Theorem 3.3]{Bar18}.

\begin{theorem}\label{Ex5}
Let $\Gamma\subset\SU(n,1)$ be a discrete subgroup acting effectively
and satisfying
$\text{Vol}(B_n/\Gamma)<+\infty$,
let $g_0\in\Gamma$ be a loxodromic element with real eigenvalues
and let $\Gamma_0\subset\Gamma$ be the subgroup
generated by $g_0$. Let $\wi\gamma:\R\rightarrow B_n$
be the geodesic generated by $g_0$, and let
$x,\,y\in\partial B_n$ be the two points of the boundary
of $B_n$ joined by $\wi\gamma$, and assume that $x,\,y\in\R^n$.

Then there exists $p_0\in\N^*$ such that  for
all $p\geq p_0$, the convergent series
\begin{equation}\label{RPSlox}
s_{\gamma,p}(z)=\sum_{[g]\in\Gamma_0\backslash\Gamma}
\mathfrak{J}(g,z)^{-2(n+1)p}\frac{1}
{(\langle g.z,x\rangle\langle g.z,y\rangle)^{(n+1)p}}
\,\sigma_z^p\;,
\end{equation}
does not vanish identically in
$z\in B_n$.
\end{theorem}
\begin{proof}
Consider the setting of Section \ref{s5.2}, with $\wi{X}=B_n$
and $\wi{L}=K_{B_n}^2$, so that $X=B_n/\Gamma$ and
$L=K_X^2$. Write
$\gamma:S^1\rightarrow X$
for the geodesic loop induced by
$\wi\gamma:\R\rightarrow\wi{X}$ in the quotient,
so that $(S^1,\gamma,\zeta)$ is a compact Bohr-Sommerfeld submanifold
 of $(X,K_X^2)$ by the result of Barron in \cite[\S\,2.2]{Bar18}.
Then as explained in the last paragraph of Section \ref{s5.2},
we can apply Theorem \ref{t5.6} to get $p_0\in\N$ such that the
associated section $\wi{s}_{\gamma,p}\in H^0(B_n,K_{B_n}^{2p})$
defined by \eqref{rPsgalfla} does not vanish for $p\geq p_0$.

On the other hand, by
\cite[Theorem 14]{FK01} and the explicit formula given in
\cite[\S\,6.3]{FK01}, the section of $H^0_{(2)}(X,K_X^{2p})$
induced by the formula \eqref{RPSlox} for all $p\in\N^*$
satisfies the
characterizing property \eqref{rep} for $(S^{1},\gamma,\zeta)$,
up to an explicit multiplicative constant $C_p>0$.
By Proposition \ref{proprepgal}, we deduce that
this sections is equal to $C_p\,\wi{s}_{\Lambda,p}$, and thus
does not vanish identically.
\end{proof}

If furthermore $n=2$ and for a loxodromic element $g_0\in\Gamma$
with real eigenvalues in $\SU(2,1)$,
Barron considers in \cite[\S\,3,\,\S\,4]{Fot02} relative Poincar\'e series 
associated with some remarkable Lagrangian tori,
and computes them explictly. The following result show that
these series do not vanish for $p\in\N^*$ large enough.

\begin{theorem}\label{Ex6}
Let $\Gamma\subset\SU(2,1)$ be a discrete subgroup acting effectively
and satisfying $\text{Vol}(B_2/\Gamma)<+\infty$,
let $g_0\in\Gamma$ be a loxodromic element 
with real eigenvalues and let $\Gamma_0\subset\Gamma$ be 
the subgroup generated by $g_0$.
Let $\wi\gamma:\R\rightarrow B_2$ be the geodesic generated
by $g_0$, let $x,\,y\in\partial B_2$ be the two points of the
boundary of $B_2$ joined by $\wi\gamma$ and let
$v\in\C^n$ be the eigenvector of $g_0$ corresponding
to the eigenvalue $1$.

Then for any $l\in\N^*$, there exists $p_0\in\N^*$ such that for
all $p\geq p_0$, the convergent series
\begin{equation}\label{rPstori}
\sum_{[g]\in\Gamma_0\backslash\Gamma}\mathfrak{J}(g,z)^{-2p}
\frac{\langle z,v\rangle^{2l}}
{(\langle z,x\rangle\langle z,y\rangle)^{3p+l}}\;
,
\end{equation}
does not vanish identically in
$z\in B_2$.
\end{theorem}
\begin{proof}
For any $l\in\N^*$ and $\Gamma_0\subset\SU(2,1)$ generated by
a loxodromic element with real eigenvalues as above,
Barron constructs in \cite[\S\,4]{Fot02} a Lagrangian submanifold
$\wi\Lambda_l$ of $B_2$ satisfying the Bohr-Sommerfeld condition
for $(K_{B_2}^2,h^{K_{B_2}^2})$,
and shows in \cite[Proposition 4.7]{Fot02} that \eqref{rPstori}
corresponds up to a multiplicative constant to the isotropic state
associated with the Lagrangian torus
$\Lambda_l:=\wi\Lambda_l/\Gamma_0$
using the reproducing property of Proposition \ref{proprepgal}.
From Theorem \ref{t5.6}, we conclude that for any $l\in\N^*$,
there exists $p_0\in\N^*$
such that \eqref{rPstori} does not vanish identically in
$z\in B_n$ for any $p\geqslant p_0$.
\end{proof}

In the case of a smooth, compact manifold $(X,g^{TX})$,
this result was derived in \cite[\S\,4]{Fot02}, using
\cite[Theorem 3.2]{BPU95} and the fact that $\wi\Lambda_l$ is Lagrangian.
\end{example}


\def\cprime{$'$}

\end{document}